\documentclass[12pt, reqno]{amsart}

\usepackage{amsmath}
\usepackage[normalem]{ulem}

\usepackage{amssymb}
\usepackage{mathrsfs}
\usepackage{amscd}
\usepackage[usenames,dvipsnames,svgnames,table]{xcolor}

\numberwithin{equation}{section}

\newif\ifdraft\drafttrue

\draftfalse

\long\def\combarak#1{\ifdraft{\marginpar{\sn
#1 \ (BW)}}\else\ignorespaces\fi}

\font\sn = cmssi8 scaled \magstep0

\newcommand\E{\mathbf{e}}

\newcommand\name[1]{\label{#1}{\ifdraft{\sn [#1]}\else\ignorespaces\fi}}

\newcommand\eq[2]{{\ifdraft{\ \tt [#1]}\else\ignorespaces\fi}\begin{equation}\label{eq:#1}{#2}\end{equation}}

\newcommand {\equ}[1]     {\eqref{eq:#1}}

\newcommand\Disc{{\mathrm{disc}}}

\newcommand{\Q}{{\mathbb {Q}}}

\newcommand{\R}{{\mathbb{R}}}

\newcommand{\T}{{\mathbb{T}}}

\newcommand{\Z}{{\mathbb{Z}}}

\newcommand{\C}{{\mathbb{C}}}

\newcommand{\N}{{\mathbb{N}}}

\newcommand{\A}{{\mathbf{a}}}

\newcommand{\B}{{\mathbf{b}}}

\newcommand{\diam}{\operatorname{diam}}

\newcommand {\ignore}[1]  {}

\newcommand{\df}{{\, \stackrel{\mathrm{def}}{=}\, }}

\newcommand{\til}{\widetilde}

\newcommand{\sm}{\smallsetminus}

\newcommand{\vre}{\varepsilon}

\font\comment = cmbx10 scaled \magstep0

\newtheorem{thm}{Theorem}[section]
\newtheorem{lem}[thm]{Lemma}
\newtheorem{prop}[thm]{Proposition}
\newtheorem{cor}[thm]{Corollary}

\newtheorem{remark}[thm]{Remark}

		  \newcommand{\sgn}{\text{sgn}}

		  \newcommand{\mc}{\mathcal}

		  \newcommand{\rar}{\rightarrow}

		  \newcommand{\dsum}{\displaystyle\sum}
		  
		  \newcommand{\bx}{\mathbf{x}}
		  \newcommand{\bt}{\mathbf{t}}
		  
		  \newcommand{\bm}{\mathbf{m}}
		  \newcommand{\by}{\mathbf{y}}
		  \newcommand{\bu}{\mathbf{u}}
		  		  
	  \newcommand{\spa}{\mathrm{span}}
		  \newcommand{\msG}{\mathscr{G}}
		  \newcommand{\msF}{\mathscr{F}}
		  \newcommand{\lb}{\left\lbrace}
		  \newcommand{\rb}{\right\rbrace}
		  \newcommand{\bxi}{\boldsymbol{\xi}}

\begin{document}
\title[Equivalence relations on separated nets]{Equivalence relations
  on separated nets arising from linear toral flows}
\author{Alan Haynes}
\address{School of Mathematics, University of York, York,   UK }
\email{ alan.haynes@york.ac.uk}

\author{Michael Kelly}
\address{
Dept. of Mathematics, University of Texas, Austin, TX USA
}
\email{ mkelly@math.utexas.edu}

\author{Barak Weiss}
\address{Dept. of Mathematics, Tel Aviv University, Tel Aviv, Israel }
\email{barakw@post.tau.ac.il}

\maketitle
	\begin{abstract}
In 1998, Burago-Kleiner and McMullen independently proved the
existence of separated nets in $\R^d$ which are not bi-Lipschitz
equivalent (BL) to a lattice. A finer equivalence relation than BL
is bounded displacement (BD). Separated nets arise
naturally as return times to a section for minimal $\R^d$-actions. We analyze the separated nets which
arise via these constructions, focusing particularly on nets arising
from linear $\R^d$-actions on tori. We show that generically these
nets are BL to a lattice, and for some choices of dimensions and sections, they are
generically BD to a lattice. We also show the existence of such nets
which are not BD to a lattice.
	\end{abstract}

\section{Introduction}
A {\em separated net} in $\R^d$ is a subset $Y$
for which there are $0<r<R$
such that any two distinct points of $Y$ are at least a distance $r$
apart, and any ball of radius $R$ in $\R^d$ contains a point of
$Y$.  Separated nets are sometimes referred to as {\em Delone
  sets}.
%
 The simplest example of a separated net is a lattice in $\R^d$,
and it is natural to inquire to what extent a given separated net
resembles a lattice. To this end we define equivalence relations
on separated nets: we say that $Y_1, Y_2$ are {\em bi-Lipschitz
  equivalent}, or {\em BL}, if there is a bijection $f: Y_1\to Y_2$ which is
bi-Lipschitz, i.e. for some $C>0$,
$$
\frac1C \|{\bf x}-{\bf y}\|  \leq \|f({\bf x}) - f({\bf y})\| \leq C \|{\bf x}-{\bf y}\|
$$
for all ${\bf x},{\bf  y} \in Y_1$; we say they are {\em bounded displacement}, or
{\em BD}, if there is a bijection $f: Y_1 \to
Y_2$ for which
\eq{eq: defn BD}{
\sup_{{\bf y} \in Y} \|f({\bf y}) - {\bf y}\| < \infty.
}
It is not hard to show that for separated nets, BD implies
BL. Moreover it follows from the Hall marriage lemma (see Proposition
\ref{prop: linear maps do not change}) that
all lattices of the same covolume are in the same BD class, and hence
all lattices are in the same BL class. A fundamental
result in this context was the discovery in 1998 (by
Burago-Kleiner \cite{BK1} and McMullen \cite{McMullen}) that there are
separated nets which are not BL to a lattice. 

A simple way to construct separated nets is via an
$\R^d$-action. Namely, suppose $X$ is a compact space, equipped
with a continuous action of $\R^d$. We denote the action by $\R^d
\times X \ni ({\bf v}, x) \mapsto {\bf v}.x \in X.$
Now given $x \in X$ and a subset $\mathcal{S} \subset X$,
we can define the `visit set'
\eq{eq: visit set}{
Y = Y_{\mathcal{S}, x} \df \{ {\bf v} \in \R^d: {\bf v}.x
\in \mathcal{S} \}.
}
It is easy to impose
conditions on $\mathcal{S}$ guaranteeing
that $Y$ is a separated net for all $x$. For example, this will hold if
$X$ is a $k$-dimensional manifold,
$\mathcal{S}$ is a
 {\em Poincar\'e section} (i.e., an embedded submanifold of dimension $k-d$ everywhere
transverse to orbits) and the
$\R^d$-action is {\em minimal} (i.e. all orbits are dense). \newline
\indent  We will sometimes fix the set $\mathcal{S}$ without reference to a specific $\R^d$-action (in fact, there are times when we will vary the $\R^d$-action).  Consequently, we use the terminology `section' rather loosely, a section $\mathcal{S}$ is simply a subset of $X$.  On the other hand, we are often in the situation where we fix the action and vary the section. When we find ourselves in the latter situation, the sections we consider are essentially Poincar\'e sections. See \S \ref{subsection: sections} for further discussion on sections.

The net $Y$ obviously depends on the dynamical system $X$ chosen. We will focus
on what is perhaps the simplest nontrivial case, namely when $X = \T^k
\df \R^k/\Z^k$ is the standard $k$-torus, and
$\R^d$ acts linearly. That is,
denoting $\pi: \R^k \to \T^k$ the canonical homomorphism, and letting $V \cong \R^d$ be a $d$-dimensional linear
subspace of $\R^k$, the action is given by
\eq{eq: toral linear action}{
{\bf v}.\pi(\bx) = \pi({\bf v}+\bx).
}
In this context we will say
that $Y$ is a {\em toral dynamics separated net}, with {\em associated dimensions}
$(d,k)$.
 We remark that the toral dynamics
separated nets are intimately connected to the well-studied {\em cut-and-project}
constructions of separated nets. We briefly discuss this connection
in \S \ref{subsection: tilings}, and refer the reader to
\cite{Meyer, Senechal, MQ} for more information.

 Note that the separated net $Y$ depends
nontrivially on the choices of the subspace $V$, the section
$\mathcal{S}$, and the orbit $V.\bx$. We will be interested in
{\em typical}
toral dynamical nets; e.g. this might mean randomly choosing the
acting subspace $V$ in the relevant Grassmannian variety\footnote{Throughout this work we often make {\it `almost everywhere'} statements without specific reference to a measure or measure class. When such a situation is encountered, we will understand {\it `a.e.'} with respect to the {\it smooth measure class}.}, and/or
the section $\mathcal{S}$ in a finite dimensional set of shapes such
as parallelotopes, etc. We remark (see \S \ref{subsection: sections}) that different choices of $x$ do
not have a significant effect on the properties of $Y$.

\ignore{
Another simple way to construct a separated net is via tilings of
Euclidean space. Recall that a tiling of $\R^d$ is its subdivision
into countably many regions homeomorphic to closed balls, with
disjoint interiors, such that the regions are isometric to finitely
many possible sets called {\em basic tiles}. If we mark one point on
each basic tile, the set of images of those points, one in each region
of our tiling, forms a separated net, whose BDD class depends only on
the tiling. Two classical constructions of tilings, namely the {\em
  cut and project} and {\em substitution} tilings recalled in \S
\ref{subsection: tilings}, give rise to interesting and well-studied
aperiodic tilings, such as the Penrose tiling. As we will explain, the
separated nets arising from the
cut and project tilings are special cases of toral dynamics separated
nets.
}

The constructions of \cite{BK1, McMullen} were rather indirect,
and left open the question of whether any of the nets constructed via
toral dynamics
is equivalent (in the sense of either BL or BD) to a lattice. In
 \cite{BK2}, Burago
 and Kleiner addressed this issue, and showed that a typical
toral dynamics separated net with associated dimensions (2,3) is
BL to a lattice. We analyze the situations in arbitrary dimensions
$(d,k)$. 

In order to state our main results we will need some definitions.  Given a $\R^d$-action, a section $\mathcal{S}$ is {\it bounded} if it is the image of a bounded subset $B\subset \R^{k-d}$ through a smooth injective map which is everywhere transverse to orbits and extends to the closure of $B$.  We will say that a section $\mathcal{S} \subset \T^k$
is {\em linear} if it is the image under $\pi$ of a bounded subset $B$
of a $(k-d)$-dimensional plane transverse to $V$. A section has nonempty $(k-d)$-dimensional interior if the corresponding set $B$ has nonempty interior. We will use the
notation $\dim_M(\partial \mathcal{S})$ to denote the upper Minkowski dimension of the boundary of $\mathcal{S}$, a notion we
recall in \S \ref{section: prelims}. \newline
\indent  Our first result shows that being BL to a lattice is quite
common for toral dynamics nets:

\begin{thm}\name{thm: main, BL}
For a.e. $d$-dimensional subspace $V \subset \R^k$, for any $\bx \in
\T^k$, and any section $\mathcal{S}$
which is bounded and has nonempty $(k-d)$-dimensional interior,
and satisfies $\dim_M \, \partial \mathcal{S} <
(k-d)$, the corresponding separated net is BL
to a lattice.
\end{thm}
It would be interesting to know whether there is a toral dynamics separated net
which is {\em not} BL to a lattice.

Our second result deals with the equivalence relation BD. Here the situation is more
delicate, and we have the following:

\begin{thm}\name{thm:  main, BDD}
Consider toral dynamics nets with associated dimensions $(d,k)$.
\begin{enumerate}
\item
If $(k+1)/2<d < k$, then for
a.e. $V$, any $\bx \in \T^k$, and
linear section $\mathcal{S}$
which is $(k-d)$-dimensionally open,
and satisfies $\dim_M \, \partial \mathcal{S} =
k-d-1$, the corresponding
separated net is
BD to a lattice.
\item
For any $2 \leq d < k$,  for a.e.
$V$, for any $x \in \T^k$ and any linear section
$\mathcal{S}$ which is a box with sides parallel to $k-d$ of the
coordinate axes, the corresponding net is BD to a lattice.
\item
For a.e. linear section $\mathcal{S} \subset B$ which is a
parallelotope, there is a residual set of subspaces $V$ for
which the corresponding net is not BD to a lattice.
\end{enumerate}
\end{thm}

Our strategy of proof is inspired by \cite{BK2, Solomon, DSS}. We
use work of Burago-Kleiner \cite{BK2} and Laczkovich \cite{Laczkovich}
to relate the notions of BL and BD to rates of convergence of some
ergodic averages for our toral $\R^d$-action. This rate of convergence
is studied via harmonic analysis on $\T^k$, and leads to the study of
Diophantine properties of the acting subspace
$V$.
 The connection between Diophantine properties of $V$ and rates of
convergence of ergodic averages on $\T^k$ is standard and well-studied
in the literature on
discrepancy, see e.g. \cite{DT}. However none of the existing results
in the literature supplied the estimates we needed.
Before
stating our results in this direction, we introduce some notation.

We will use boldface letters such as $\mathbf{v,x}$ to denote vectors
in $\R^k$, and denote their inner product by $\mathbf{v} \cdot
\mathbf{x}$. 
Let $V = \spa \left(\mathbf{v}_1, \ldots, \mathbf{v}_d\right).$
For $T > 0$ we set
\eq{eq: defn BT}{
B_T \df \left\{\sum a_i \mathbf{v}_i : \max_i |a_i| \leq T\right\}.
}
The notation $|A|$ denotes the Lebesgue measure of a measurable set
$A$ in $\R^k$ or $\T^k$.
Given $U \subset \T^k$, $T\ge 0$ and ${\bf x}\in\R^k$ we set
\[N_T(U,{\bf x}) \df \int_{B_T}\chi_U\left(\pi({\bf x}+\bf{t})\right)~d{\bf t}.\]
The reader should note that this notation suppresses the dependence of
$N_T(U, \mathbf{x})$ on the choice of the subspace $V$ as well as the
basis $\mathbf{v}_1, \ldots, \mathbf{v}_d$.
We will denote by $\|\mathbf{m}\|$ the sup-norm of a vector
  $\mathbf{m} \in \R^k$, and
say that $\mathbf{v}$ is {\em Diophantine} if there are positive
constants
$c, s$ such that
\eq{eq: defn Diophantine vector}{
|{\bf m}\cdot {\bf v}|\ge \frac{c}{\|{\bf
  m}\|^{s}}~\text{, \ for all nonzero}~{\bf m}\in\Z^k.
}
We will say that $V$ is {\em Diophantine} if it contains a Diophantine
vector.

By an {\em aligned box} in $\T^k$ we mean the image, under $\pi$, of a set of
the form $[a_1, b_1] \times \cdots  \times [a_k, b_k]$ (a box with
sides parallel to the coordinate axes), where $b_i-a_i <1 $ for all
$i$ (so that $\pi$ is injective on the box).

\begin{thm}\name{thm: equidistribution for BL}
Suppose $V$ is Diophantine. Then
there are constants $C$ and $\delta>0$ such that for any
${\bf x}\in\R^k$, any $T>1$, and any aligned box $U \subset \T^k$,
\eq{eq: equidistribution}{
\Big|N_T(U,{\bf x})-|U||B_T|\Big|\leq C T^{d-\delta}.
}

\end{thm}
We remark that under a stronger Diophantine assumption, which still holds for
almost every subspace $V$, conclusion
\equ{eq: equidistribution} can be strengthened, replacing $T^{d-\delta}$ with $(\log
T)^{k+2d+\delta}$. See Proposition \ref{prop: strengthening}.








Given a basis $\mathcal{T} =( \mathbf{t}_1, \ldots, \mathbf{t}_k)$ of
$\R^k$, we denote
\eq{eq: defn r general}{
r_{\mathcal{T}}(\mathbf{m}) \df \prod_{i=1}^k \min \left(1, \frac{1}{|\mathbf{t}_i
    \cdot \mathbf{m}|} \right),
}
and say
that $\mathbf{v}_1, \ldots, \mathbf{v}_d$ are {\em strongly
  Diophantine (with respect to $\mathcal{T} $)} if for
any $\vre>0$ there is $C>0$ such that for any $M>0$,
\eq{eq: for strongly Diophantine}{
\sum_{\substack{\mathbf{m} \in \Z^k \sm \{0\} \\ \| \mathbf{m}
    \| \leq M}} r_{\mathcal{T}}(\mathbf{m}) \prod_{i=1}^d \frac{1}{|\mathbf{m} \cdot \mathbf{v}_i|}
\leq C M^{\vre}.}
We say that $U \subset \T^k$ is {\em a
  parallelotope aligned
  with $\mathcal{T}$} if there are 
positive $b_1, \ldots, b_k$, and $\mathbf{x} \in \R^k$, such that $U =
\pi(\til U+\mathbf{x})$, where 
$$\til U \df \left\{\sum_{i=1}^d a_i \mathbf{t}_i :\forall i, \, a_i
  \in [0, b_i]\right\},
$$
and $\pi$ is injective on $\til U + \mathbf{x}$. 
Let $\mathbf{e}_1, \ldots, \mathbf{e}_k$ be the standard basis for
$\R^k$.

\begin{thm}\name{thm: equidistribution for BDD}
Suppose $\mathcal{T} = (\mathbf{v}_1, \ldots, \mathbf{v}_k)$ is a basis for $\R^k$
such
that $\mathbf{v}_i  \in \{\mathbf{e}_1, \ldots, \mathbf{e}_k\}$ for
each $i=d+1, \ldots, k$, and $\mathbf{v}_1, \ldots, \mathbf{v}_d$
is strongly Diophantine with respect to $\mathcal{T}$. 
Then for any
$\delta>0$  there is $C>0$ such that for all 
$\bx
\in \T^k$, and any $U$ which is a parallelotope aligned with
$\mathcal{T}$, with sidelengths bounded above by  $\eta$,
we have
\eq{eq: equidistribution2}{
\Big|N_T(U, {\bf x})-|U||B_T|\Big|\leq C\left( 1+\eta\right)^{k} T^\delta.}
\end{thm}

As above, we will show in Proposition \ref{prop: strengthening} that
there is a stronger Diophantine hypothesis, which still holds for
almost every $V$, under which $T^\delta$ in \equ{eq:
  equidistribution2} can be replaced
by $(\log T) ^{k+2d+\delta}$. 

\medskip
Justifying the terminology, we will see in \S \ref{section: reduction}
that a subspace with a strongly Diophantine basis is
Diophantine.
We will also see that almost every choice of $V$ (respectively
$\mathcal{T}$) satisfies the
Diophantine properties which are the hypotheses of Theorem \ref{thm:
  equidistribution for BL} (resp., Theorem 
\ref{thm: equidistribution for BDD}).
The conclusions of Theorems \ref{thm: main,
  BL} and \ref{thm: main, BDD} hold for these choices.

\medskip

Besides the cut-and-project method, another well-studied construction
of a separated net is the {\em substitution system}
construction, and results analogous to ours have appeared for
separated nets arising via substitution systems in recent work of
Solomon \cite{Solomon, Solomon2} and Aliste-Prieto, Coronel and
Gambaudo \cite{ACG}. Briefly, it was shown in these papers that all
substitution system separated nets are BL to lattices and many but not
all are BD to lattices. A particular case of interest is the {\em
  Penrose net}
obtained by placing one point in each tile of the Penrose aperiodic
tiling of the plane. 
Penrose net admits alternate descriptions via both the cut-and-project
and substitution system constructions. Using the latter approach,
Solomon \cite{Solomon} showed that that the Penrose net is BD to a
lattice.

\subsection{Organization of the paper}
In \S \ref{section: basics} we review basic material relating sections
for minimal
flows and separated nets, and the relation to cut-and-project
constructions. In \S \ref{section: BK} we state the results of
Burago-Kleiner and Laczkovich, and use these to connect the
properties of the separated net to quantitative equidistribution statements for
flows. In \S\ref{section:
  prelims} we discuss Minkowski dimension and show how to
approximate a section by aligned boxes if the Minkowski dimension of
the boundary is strictly smaller than $d$.
The main result of \S\ref{section: trig} is
Theorem \ref{thm:extremal}, which provides
good approximations to the indicator
function of parallelotopes in $\T^k$ by trigonometric polynomials.
We believe this result will be helpful for other problems in
Diophantine approximation and ergodic theory of linear toral flows.
In \S\ref{section: ET} we deduce an Erd\H{o}s-Tur\'an type inequality from
Theorem \ref{thm:extremal} and apply it to prove Theorems
\ref{thm: equidistribution for BL} and
\ref{thm:
  equidistribution for BDD}.
In \S \ref{section: reduction} we adapt arguments of W. Schmidt to
show that our Diophantine conditions are satisfied almost surely, and deduce Theorem \ref{thm: main, BL} and parts (1) and
(2) of Theorem \ref{thm: main, BDD} in \S\ref{section: proof BL}.
In
\S \ref{section: irregularities} we prove Theorem \ref{thm: main,
  BDD}(3).
\subsection{Acknowledgements}
We are grateful to Bruce Kleiner for suggesting the problem, and to
William Chen, Eric Latorre Crespo, Charles Radin,
Wolfgang Schmidt, Yaar Solomon and Jeffrey Vaaler for useful discussions. The authors
thank York University and Ben Gurion University for facilitating
mutual visits. The research of the first author was supported by EPSRC
grants EP/F027028/1 and EP/J00149X/1.
The research
of the third author was supported by Israel Science Foundation grant number
190/08 and ERC starter grant DLGAPS 279893.
\section{Basics}\name{section: basics}
\subsection{Bounded displacement}
We first recall the following well-known facts.
\begin{prop}\name{prop: linear maps do not change}
Any two lattices of the same covolume are BD to each other. Moreover, if
$Y \subset \R^d$ is BD to a lattice, and $T: \R^d \to \R^d$
is a linear isomorphism, then $T(Y)$ is also BD to a lattice.
\end{prop}

\begin{proof}
Suppose $L_1$
and $L_2$ are lattices of the same covolume $\lambda$, and define a bipartite
graph $G$ whose vertices are
the points of $L_1 \cup L_2$, and  $\bx_1 \in L_1, \,  \bx_2 \in L_2$ are joined by an edge if
$\|\bx_1 - \bx_2\| \leq r_1+r_2$, where $r_i$  is the
diameter of a bounded fundamental domain for $L_i$. To verify
the conditions of the Hall marriage lemma
\cite{hall}, let $D_2$ be a fundamental domain for $L_2$, so that
$\R^d = \bigsqcup_{\by \in L_2} \by+D_2$, and let $A \subset L_1$ with $N
\df \# \, A$. Let $F$
denote the set of points in $\R^d$ which are within a distance $r_1$
from points of $A$. Then $F$ contains at least $N$ copies of a
fundamental domain for $L_1$ so has volume at least $N
\lambda$. Therefore $F$ intersects at least
$N$ of the sets $\{\by +  D_2 : \by \in L_2\}$. By the definition of $G$,
if $F$ intersects $\by+D_2$ then $\by$ is connected to an element of $A$
by an edge. This implies that the
number of neighbors of $A$ is at least $N$.
By an infinite variant of the marriage lemma and Schr\"oder-Bernstein theorem (see \cite{McMullen}, the proof of Theorem 4.1) we obtain our bijection.

Now suppose
$L$ is a lattice in $\R^d$ and $\phi: Y \to L$ is a bijection
moving points a uniformly bounded amount, then $T \circ \phi \circ
T^{-1}$ is a bijection $T(Y) \to T(L)$ and it moves points a bounded
amount because $T$ is Lipschitz. This proves the second assertion.
\end{proof}

\subsection{Sections and minimal actions}\name{subsection: sections}
A standard technique for studying flows was introduced by
Poincar\'e. Suppose $X$ is a manifold with a flow, i.e. an action of
$\R$. Given an embedded submanifold $\mathcal{S}$ transverse to the orbits, we can study the
return map to $\mathcal{S}$ along orbits, and in this way reduce the
study of the $\R$-action to the study of a $\Z$-action.
We will be
interested in a similar construction for the case of an $\R^d$-action,
$d>1.$ Namely, given a space $X$ equipped with an $\R^d$-action, we
say that $\mathcal{S} \subset X$ is a {\em good section}
if there are bounded neighborhoods $\mathcal{U}_1, \mathcal{U}_2$ of $\bf 0$ in $\R^d$,
such that for any $x \in X$:
\begin{itemize}
\item[(i)]
there is at most one ${\bf u} \in \mathcal{U}_1$ such that
${\bf u}.x \in \mathcal{S}$.

\item[(ii)]
there is at least one ${\bf u} \in
\mathcal{U}_2$ such that
${\bf u}.x \in \mathcal{S}$.
\end{itemize}

These conditions immediately imply that the set
$Y= Y_{\mathcal{S}, x}$ of visit times
defined in \equ{eq: visit set} is a separated net; moreover the
parameters $r, R$ appearing in the definition of a separated net may
be taken to be the same for all $x \in X$, since they depend only on
$\mathcal{U}_1, \mathcal{U}_2$ respectively.

The action is called {\em minimal} if there are no proper invariant
closed subsets of $X$, or equivalently, if all orbits are dense.
The following proposition shows that good sections always exist for
minimal actions on manifolds:

\begin{prop}\name{prop: minimal}
Suppose $X$ is a compact $k$-dimensional manifold equipped with a minimal
$\R^d$-action, and suppose $\mathcal{S} \subset X$ is the image of an
open bounded $\mathcal{O} \subset \R^{k-d}$ under a smooth injective map which is everywhere
transverse to the orbits and extends to the closure of $\mathcal{O}$. Then
$\mathcal{S}$ is a good section.

\end{prop}

\begin{proof}
Since $\mathcal{S}$ is transverse to orbits, for every $x \in
\overline{\mathcal{S}}$ there is a bounded neighborhood $U=U_x$ of identity in $\R^d$ so
that for ${\bf u} \in U \sm \{{\bf 0}\}
$, ${\bf u}.x \notin \overline{\mathcal{S}}$. Since $\mathcal{O}$ is bounded, a
compactness argument shows that $U$ may be taken to be independent of
$x$, and we can take $\mathcal{U}_1$ so that $\mathcal{U}_1 -
\mathcal{U}_1 = \{\bx-\by: \bx, \by \in \mathcal{U}_1\} \subset U$, which immediately implies (i). Let
$$\widehat{\mathcal{S}} \df \{{\bf u}.s : {\bf u} \in U, s \in \mathcal{S} \}.$$
Then $\widehat{\mathcal{S}}$ is open in $X$. By a standard fact from
topological dynamics (see e.g. \cite{auslander}), the set of return
times
$$
\{{\bf u} \in \R^d: {\bf u}.x \in \widehat{\mathcal{S}} \}
$$
is {\em syndetic}, i.e. there is a bounded set $K$ such that for
any ${\bf w} \in \R^d$, there is ${\bf k} \in K$ with $({\bf w}+{\bf k}).x \in
\widehat{\mathcal{S}}$. By minimality this implies that for any $x \in X$
there is ${\bf k} \in K$ such that ${\bf k}.x \in \widehat{\mathcal{S}}$. Taking
$\mathcal{U}_2 = K-U$ we obtain (ii).
\end{proof}

If $X$ is not minimal, there will be some $x$ and $\mathcal{S}$
 for which $Y$ is not syndetic. However good sections exist for any
 action:
\begin{prop}\name{prop: sections exist}
For any action of $\R^d$ on a compact manifold, there are good
sections.

\end{prop}
\begin{proof}
Fix a bounded symmetric neighborhood $\mc{U}$ of $\bf 0$ in $\R^d$. We can
assume that $\mc{U}$ is sufficiently small, so that for each $x
\in X$ there is an embedded submanifold $\mc{S}_x$ of dimension $k-d$
such that the map
$$\mathcal{U} \times  \mathcal{S}_x \to
X, \ \ ({\bf u},x) \mapsto {\bf u}.x
$$
is a diffeomorphism onto a neighborhood $\mathcal{O}_x$ of $x$. By
compactness we can choose $x_1, \ldots, x_r$ so that the sets
$\mathcal{O}_{j} = \mc{O}_{x_j}$ are a cover of $X$.
By a small perturbation we can
ensure that the closures of the $\mathcal{S}_j = \mc{S}_{x_j}$ are disjoint. Let $\mathcal{S} = \bigcup_j
\mathcal{S}_j$, then it is clear by construction that (ii) holds for
$\mathcal{U}_2 = \mc{U}$. Since the $\mathcal{S}_j$ are disjoint, a
compactness argument shows (i).
\end{proof}

In both Propositions \ref{prop: minimal} and \ref{prop: sections exist}, the section we constructed was
the image of an open bounded subset $\mathcal{O} \subset \R^{k-d}$
under a smooth injective map which extends to the boundary of
$\mathcal{O}$. We call sections arising in this manner {\em
  $(k-d)$-dimensionally open and bounded.} 

The following will be useful when we want to go from a section to a
smaller one.

\begin{prop}\name{prop: replace with rational subspace}
Suppose $\mathcal{S}$ is a section for an $\R^d$ action on a space
$X$, $x \in X$, and
$\mathcal{S}= \bigsqcup_{i=1}^r \mathcal{S}_i$ is a partition into
subsets.  Suppose that for $i=1, \ldots, r$, each $Y_i \df
Y_{\mathcal{S}_i,x} $ is BD to a fixed lattice $L$. Then $Y_{\mathcal{S},x}$ is BD to a lattice.

\end{prop}

\begin{proof}
Clearly $Y_{\mathcal{S},x} = \bigsqcup_1^r Y_i$, and by
assumption, for each $i$ there is a bijection $f_i : Y_i \to
L$ moving points a bounded distance. Let $\hat{L}$ be a lattice containing $L$ as a subgroup of index
$r$ and let ${\bf v}_1, \ldots, {\bf v}_r$ be coset representatives for
$\hat{L}/L$.
Then
$$
f({\bf y}) = f_i({\bf y})+{\bf v}_i \ \ \ \mathrm{for \ } {\bf y} \in Y_i
$$
is the required bijection between $Y$ and $\hat{L}$.
\end{proof}

\begin{prop}\name{prop: changing section}
Suppose $\mathcal{S}_1, B$ are two good sections for an
$\R^d$-action on a space $X$. Let $\mathcal{U}_1, \, \mathcal{U}_2, \,
\mathcal{U}_1', \, \mathcal{U}_2'$ be the corresponding sets as in (i)
and (ii), for $\mathcal{S}_1$ and $B$ respectively, and assume that
\eq{eq: hypothesis 1}{
\mathcal{U}_2' - \mathcal{U}_2' \subset
  \mathcal{U}_1.
}
Then there is $\mathcal{S}_2 \subset B$,
a good section for  the action, such that for each $x \in X$, the nets
$Y_i =
Y_{\mathcal{S}_i,x}$ as in \equ{eq: visit set} ($i=1,2$) are BD to
each other.
\end{prop}

\begin{proof}
For each $x \in X$, let
$u_x \in \mathcal{U}_2'$ be such that ${\bf u}_x.x \in B$.
Let $\mathcal{S}_2 \df \{{\bf u}_x.x : x \in
\mathcal{S}_1\}$. First note that $\mathcal{S}_2$ is a good
section:
$\mathcal{U}''_2 \df \mathcal{U}_2 + \{{\bf u}_x: x \in
\mathcal{S}_1\}$ satisfies (ii) for $\mathcal{S}_2.$ Since
$\mathcal{U}'_1$ satisfies (i) for $B$, it also satisfies
(i) for $\mathcal{S}_2$.

Let $Y_i = Y_{\mathcal{S}_i, x} \ (i=1,2)$. It remains to show that the
$Y_i$ are BD. For each ${\bf u} \in Y_1$ we have $z={\bf u}.x \in \mathcal{S}_1$ so
that $F({\bf u}).x \in \mathcal{S}_2,$ where $F({\bf u})={\bf u}+{\bf u}_z$ and ${\bf u}_z \in
\mathcal{U}_2'$.  Clearly
$F$ moves all points a bounded distance, and maps $Y_1$ to $Y_2$. We
need to show that it is a
bijection. If ${\bf u}' \in Y_2$ then ${\bf u}'.x =s_2 \in \mathcal{S}_2$, which implies
that there is $s_1 \in \mathcal{S}_1$ with $s_2 = {\bf u}_{s_1}.s_1$. This
implies that $s_1 = ({\bf u}'-{\bf u}_{s_1}).x$ so that ${\bf u}' - {\bf u}_{s_1} \in Y_1$
satisfies $F({\bf u}'-{\bf u}_{s_1})= {\bf u}'-{\bf u}_{s_1}+{\bf u}_{s_1} = {\bf u}'$. Thus $F$ is
surjective. Now suppose ${\bf u}_1, {\bf u}_2 \in Y_1$ such that
$$
{\bf u}_1+{\bf u}_{z_1} = F({\bf u}_1)=F({\bf u}_2) = {\bf u}_2+{\bf u}_{z_2},
$$
where $z_i = {\bf u}_i x \in
\mathcal{S}_1$. Then by \equ{eq: hypothesis 1},
${\bf u}_2 - {\bf u}_1 = {\bf u}_{z_1} - {\bf u}_{z_2} \in \mathcal{U}_1$, so by (i) with $x
= s_1$, we conclude that ${\bf u}_2=
{\bf u}_1$.
\end{proof}

The nets $Y_{\mathcal{S},x}$ depend on the choice of $x$ and
$\mc{S}$. As Theorem \ref{thm: main, BDD} shows, different choices of $\mc{S}$ will
lead to very different separated nets. However, as
the following result shows, for
much of our discussion the choice of $x$ is immaterial.

\begin{prop}\name{prop: point immaterial}
Suppose $X$ is a minimal dynamical system and  $\mathcal{S}$ is a  good
section which is $(k-d)$-dimensionally open. If there is $x_0 \in X$
for which
the separated net $Y_{\mc{S},x_0}$ is BD (resp. BL) to a
lattice, then for every $x \in X$, the net $Y_{\mc{S},x}$ is also BD
(resp. BL) to a lattice.
\end{prop}

\begin{proof}
We will prove the
statement for the case of the BD equivalence relation, leaving the
other case to the reader.

Write $Y_0 \df Y_{\mc{S}, x_0}$ and $Y \df Y_{\mc{S}, x}$.
Let $\mc{L} \subset \R^d$ be a lattice and let $f : Y_0
\to \mc{L}$ be a bijection satisfying 
$$
K \df \sup_{\by \in Y} \|\by - f(\by)\| < \infty.
$$
Let $\Omega$ be a compact fundamental domain for the action of $\mc{L}$ on
$\R^d$, that is for each ${\bf z} \in \R^d$ there are unique $\boldsymbol{\ell} = \boldsymbol{\ell}({\bf z}) \in
\mc{L}, \, \boldsymbol{\omega} = \boldsymbol{\omega}({\bf z})\in \Omega$ with ${\bf z} = \boldsymbol{\ell} + \boldsymbol{\omega}.$
Let $x \in X$ and
let ${\bf u}_n \in \R^d$ such that ${\bf u}_n.x_0 \to x$. Using the continuity of the
action on $X$, and the assumption that $\mc{S}$ is
$(k-d)$-dimensionally open, it is easy to see that the translated nets
$Y_0 - {\bf u}_n$ converge to $Y$ in the following
sense. Let $B({\bf x},T)$ denote the Euclidean open ball of radius $T$
around ${\bf x}$.  For any $T>0$ for which there is no element of $Y$ of norm $T$,
and any
$\vre>0$ there is $n_0$ such that for any $n>n_0$, there is a
bijection between $B({\bf 0},T) \cap Y$ and $B({\bf 0},T) \cap (Y_0 - {\bf u}_n)$ moving
points at most a distance $\vre$.

Now
for each $k$ we take $n = n(k)$ large
enough so that for each $\by \in B({\bf 0},k) \cap Y$, there is $\bx=\bx(\by) \in Y_0- {\bf u}_n$ with
$\|\by-\bx\|< 1.$ Define $f_k: B({\bf 0},k) \cap Y\to \mc{L}$ by
$$
f_k({\bf y}) \df f(\bx(\by)+{\bf u}_n) - \boldsymbol{\ell}({\bf u}_n).
$$
Then for each $k \geq k_0 > 0,$ and each ${\bf y} \in B({\bf 0},k) \cap Y$,
\[
\begin{split}
\|\by - f_k(\by) \| & \leq \|\by-\bx(\by) \| +\|\bx(\by)+{\bf u}_n- f(\bx(\by)+{\bf u}_n)\| + \|{\bf u}_n -\boldsymbol{
\ell}({\bf u}_n) \|
\\ &\leq 1+K+\diam(\Omega)
;
\end{split}
\]
that is, points in
$B({\bf 0}, k_0) \cap Y$ are moved a uniformly bounded distance by the maps
$f_k, \, k\geq k_0$. In particular the set
of possible values of the maps $f_k(\by), \, k
\geq k_0$ is finite.
Thus by a diagonalization procedure we may choose a subset of the
$f_k$ so that for each $\by \in Y$, $f_k(\by)$ is eventually constant. We
denote this constant by $\hat{f}(\by)$. Now it is easy to check that $\hat{f}$ is a
bijection satisfying \equ{eq: defn BD}.
\end{proof}
\combarak{I do not know whether the nets $Y_{\mathcal{S}, x_0},
    Y_{\mathcal{S}, x}$ are always BD.}

We now specialize to linear actions on tori. It is known that a linear
action of a $d$-dimensional subspace $V \subset \R^k$ on $\T^k$ as in
\equ{eq: toral linear action} is
minimal if and only if $V$ is {\em totally irrational}, i.e., $\pi(V)$ is dense in $\T^k$.
Suppose $V$ is totally irrational and of dimension $d$, so that the
action of $V$ on $\T^k$ is minimal. Note that when using this action
to define separated nets via \equ{eq: visit set}, one needs to fix an
identification of $V$ with $\R^d$; however, in light of Proposition
\ref{prop: linear maps do not change},
for the questions we will be considering, this choice will be
immaterial.

 Let $W$ be
a subspace of dimension $k-d$, such that
$\R^k = V \oplus W$. For any bounded open subset
$B'$ in $W$, such that $\pi|_{\bar{B'}}$ is injective, $B \df \pi(B')$ is a good section, in
view of Proposition \ref{prop: minimal}. We do not assume
that $W$ is totally irrational, so that $\pi$ need not be globally
injective on $W$. Such sections will be called {\em linear} sections.

When discussing sections, there is no loss of
generality in considering linear sections:
\begin{cor}\name{cor: wnlg}
Let $V$ and $W$ be as  above, and assume
$W$ is totally irrational. Then for any good section $\mathcal{S}$ for the
linear action of $V$ on $\T^k$, there is a linear section
$\mathcal{S}' \subset \pi(W)$ such that for any $x \in \T^k$,
$Y_{\mathcal{S},x}$ and $Y_{\mathcal{S}',x}$ are BD.
\end{cor}
\begin{proof}
Since $W$ also acts
minimally, for any $\vre>0$, there is a sufficiently large ball $B'
\subset W$ such that
$B = \pi(B')$ is
$\vre$-dense in $\T^k$. That is, we can make the neighborhood
$\mathcal{U}_2$ appearing in (ii) as small as we wish. Thus,
given any section $\mathcal{S}$ for the action of $V$, we can make
$B'$ large enough so that \equ{eq: hypothesis 1} holds. So the claim
follows from
Proposition \ref{prop: changing section}.
\end{proof}

{\em When we say that the section
$\mathcal{S}$ is $(k-d)$-dimensionally open, bounded, is a
parallelotope, etc., we mean that $\mathcal{S}
  = \pi(\mathcal{S}')$ where $\mathcal{S}'\subset W$ has the corresponding
  properties as a subset of $W \cong \R^{k-d}$.}

\ignore{

We have chosen the $W$ to be totally irrational to ensure
that $\pi|_{B'}$ has no self-intersections.
\combarak{if we want to do so, one can prove that there is no loss of
  generality in assuming $W$ is completely arbitrary, this is by
  chopping up an arbitrary section into small pieces of equal volume
  and projecting them onto the image of $W$. If the dynamical net
  resulting from each small piece is BDD to a lattice, and the
  expansion constants are the same for each piece (which is why we
  took equal volume for each piece), then we can recover BDD of the
  original dynamical net to a lattice by ``interlacing''.}

}

\subsection{Cut and project nets}\name{subsection: tilings}
Fix a direct sum
decomposition $\R^k =V \oplus W$ into $V \cong \R^d, \, W \cong \R^{d-k}$. Let $\pi_V: \R^k
\to V$ and $\pi_W: \R^k \to W$ be the projections associated with this
direct sum decomposition. Suppose $L \subset \R^k$ is a lattice, and
$K \subset W$ is a non-empty bounded open set. The
{\em cut-and-project construction} associated to this data is
$$
\mathcal{N} = \mathcal{N}_{L,K,V,W} \df \{\bx \in V: \exists \by \in L, \pi_V(\by)=\bx, \pi_W(\by) \in
K\}.
$$
The set $\mathcal{N}$  is always a separated net in $V \cong \R^d$,
and under suitable assumptions, is 
aperiodic (e.g. is not a finite union of lattices). 
This is a particular  case of a family of more general constructions involving
locally compact abelian groups.
We refer to \cite{Senechal, MQ} for more details.

Unsurprisingly, the construction above may be seen as a toral dynamics
separated net. Since we will not be using it, we leave the proof of the following to the reader:
\begin{prop}\name{prop: toral cut and project}
Given $L$, $\R^k = V \oplus W$ and $K \subset W$ as above, there is a
linear subspace $V' \subset \R^k$, a section $\mathcal{S} \subset
\T^k$, and $x \in \T^k$, such that $\mathcal{N}_{L,K,V,W} =
Y_{\mathcal{S},x}$, where $Y_{\mathcal{S},x}$ is as in \equ{eq: visit
  set} for the action \equ{eq: toral linear action}.
\end{prop}
\qed


\ignore{
An important net arising from the cut-and-project construction is a
{\em Penrose net}, which is usually defined by placing one point in
each tile of a Penrose tiling, but can also be defined (as proved by
De Bruijn \cite[\S 7]{DeBruijn}) by the above
construction, with $k=5, \, d=2$,
$$V =\left\{\vec{v}: \sum v_j=0, \, \sum v_j
\cos\left( \frac{j \pi }{5}  \right) =0, \, \sum v_j \sin\left
  (\frac{j\pi }{5} \right) =0\right \},$$
$W = V^{\perp}$, and
$$
K = \pi_W \left([0,1]^5 \right).
$$
 (here $v_1, \ldots, v_5$ are the coefficients of $\vec{v}$).

{\comment This is wrong, fix this. }
}
\section{Results of Burago-Kleiner and  Laczkovich, and their
  dynamical interpretation} \name{section: BK}
Let $Y$ be a separated net. The question of whether $Y$ is BL or BD
to a lattice is related to the number of points of $Y$ in large sets
in $\R^d$. More precisely,
fix a positive number $\lambda$, which should be
thought of as the asymptotic density of $Y$, and
for $E \subset \R^d$,
define
$$
\Disc_Y(E, \lambda) \df \Big | \# (Y \cap E) - \lambda |E|  \Big |,
$$
where $|E|$ denotes the $d$-dimensional Lebesgue measure of $E$ (`disc'
stands for {\em discrepancy}).
If $Y$ is a lattice, and $E$ is sufficiently regular (e.g. a large
ball), then one has precise estimates showing that $\Disc_Y(E, \lambda)$ is
small, relative to the measure of $E$.
In this section we present some results which show that for arbitrary
$Y$, bounds on $\Disc_Y(E, \lambda)$ are sufficient to ensure that $Y$ is BL or
BD to a lattice.

For each $\rho \in \N$ and $\lambda>0$, let
$$
D_{Y}(\rho, \lambda) \df \sup_B \frac{\Disc_Y(B, \lambda)}{\lambda|B|},
$$
where the supremum is taken over all cubes $B \subset \R^d$ of the
form
$$B= [a_1\rho, (a_1+1)\rho] \times \cdots \times
[a_d\rho, (a_d+1)\rho],  \ \ \mathrm{with} \  a_1, \ldots, a_d \in \Z.$$

\begin{thm}[Burago-Kleiner]
If there is $\lambda>0$ for which
\eq{eq: BK need to check}{\sum_{\rho} D_Y(2^\rho, \lambda)<
\infty}
then $Y$ is BL to a lattice.
\end{thm}
\begin{proof}
The theorem was proved in case $d=2$ in \cite{BK2}, and in \cite{ACG}
for general $d$.
\end{proof}

Using this we state a dynamical sufficient condition guaranteeing that a dynamical
separated net is BL to a lattice. We will denote the Lebesgue measure
of $B \subset \R^d$ by $|B|$ and write the Lebesgue measure element
as $d\bt$. Let $\mathbf{v}_1, \ldots, \mathbf{v}_d$ be a basis of
$\R^d$ and define $B_T$ via \equ{eq: defn BT}.
Note that $|B_T| = CT^d$ for some $C>0$.
For $W \subset X$ and $x \in X$, denote
$$
N_T(W,x) \df \int_{B_T} \chi_W(\mathbf{t}.x)\, d \mathbf{t},
$$
where $\chi_W$ is the
indicator function of $W$. The asymptotic behavior of such
{\em Birkhoff integrals} as $T \to \infty$ is a well-studied topic in
ergodic theory. The action of $\R^d$ on $X$ is said to
be  {\em uniquely ergodic} if there is a measure $\mu$ on $X$ such
that for any continuous function $f$ on $X$,
and any $x \in X$,
$$\left|\int_{B_T} f(\mathbf{t}.x) \, d\mathbf{t} -|B_T|\, \int_X f \, d\mu \right| = o(|B_T|).$$
 We now show that a
related quantitative estimate implies that certain dynamical nets are
BL to a lattice.
\begin{cor}\name{cor: dynamical for BL}
Suppose $\R^d$ acts on $X$ and $\mathcal{S}$ is a good section for the
action. Let $\mathcal{U}_1$ be a neighborhood of the
identity in $\R^d$ satisfying (i) of \S \ref{subsection: sections},
and let
\eq{eq: defn V}{
W \df \{\bu.x: \bu \in \mathcal{U}_1,\, x \in \mathcal{S} \} \subset X.
}
Suppose there are
positive constants $a, C, \delta$ such that for all $x \in X$ and $T>1$,
\eq{eq: need to check, BL}{
\Big|N_T(W,x) - a|B_T|  \Big| < C\, T^{d-\delta}.
}
Then for any $x \in X$, the net $Y_{\mathcal{S},x} $, as in \equ{eq: visit set}, is BL to a
lattice.
\end{cor}
\begin{proof}
Let $x \in X,\, Y = Y_{\mathcal{S},x}$ and let $B = x'+B_T \subset \R^d$,
i.e. $B$ is a cube of side
length $2T$, with sides parallel to the coordinate hyperplanes, and
center at $x'$.
We want
to bound $\# \, Y \cap B$ in terms of $N_T(W,x')$.
Let $r$ denote the diameter of $\mathcal{U}_1$, and let
$b = |\mathcal{U}_1|$. If $\by \in Y \cap B$ then $\by.x \in
\mathcal{S}$ and hence $(\by+\bu).x \in W$ for any $\bu \in
\mathcal{U}_1$. This implies that
$$
N_{T+r}(W,x) \geq (\# \, Y \cap B) \, b.
$$
Similarly, if $\chi_W(\by.x)=1$ then there is $\by' \in Y$ with
$\|\by'-\by\|\leq r$, which implies that
$$
N_{T-r}(W,x) \leq (\# \, Y\cap B) \, b.
$$
Applying \equ{eq: need to check, BL} we find that
$$
\frac{a}{b} |B_{T-r}| -\frac{C}{b}(T-r)^{d-\delta} \leq \# \, Y \cap B
\leq \frac{a}{b} |B_{T+r}| +\frac{C}{b} (T+r)^{d-\delta}.
$$
So for any $\delta'<\delta$ there is $T_0$ such that for $T>T_0$,
setting $\lambda = a/b$ gives
$$
\Disc_Y(B_T, \lambda) \leq T^{d-\delta'}.
$$
Since $|B_T| = c T^d$ for some $c>0$, we find that $D_Y(T, \lambda) =
O(T^{-\delta'}).$ From this \equ{eq: BK need to check} follows.
\end{proof}

We now turn to analogous results for the relation BD. Our results in
this regard rely on  work of Laczkovich.
We first introduce some notation. For a measurable $B \subset \R^d$, we
denote by $|B|$ the Lebesgue measure of $B$, by $\partial \, B$
the boundary of $B$, and by $|\partial \, B|_{d-1}$ the
$(d-1)$-dimensional volume of $\partial \, B$.
By a {\em unit cube} (respectively, {\em dyadic cube}) we mean a cube
of the form
$$
[a_1, b_1) \times \cdots \times [a_k, b_k),
$$
where for $i=1, \ldots, k$ we have $a_i \in \Z$ and $b_i -a_i =1$
(respectively, $b_i-a_i = 2^j$ for a non-negative integer $j$
independent of $i$).
\begin{thm}[\cite{Laczkovich}, Theorem 1.1]\name{thm: Laczkovich2}
For a separated net $Y\subset\R^d$, and $\lambda>0$, the following are equivalent:
\begin{enumerate}
\item $Y$ is BD to a lattice of covolume $\lambda^{-1}$.
\item There is $c>0$ such that for every finite
  union of unit cubes $\mc{C}\subset\R^d$,
\begin{equation*}
\Disc_Y(\mc{C}, \lambda)
\le c \, |\partial \, \mc{C}|_{d-1}.
\end{equation*}
\item There is $ c>0$ such that
  for any measurable $A$,
$$
\Disc_Y(A, \lambda) \leq c \, \left|(\partial \, A)^{(1)} \right |,
$$
where $(\partial \, A)^{(1)}$ denotes the set of points whose distance 
from the boundary of $A$ is less than 1.
\end{enumerate}
\end{thm}

When applying this result, another result of Laczkovich is very
useful. For sets $\mc{C}, Q_1, \ldots, Q_n$, we say that $\mc{C} \in
S(Q_1, \ldots, Q_n)$ if $\mc{C}$ can be presented using $Q_1, \ldots,
Q_n$ and the operations of disjoint union and proper set difference
($A \sm B$ with $B \subset A$), with
each $Q_i$ appearing at most once. Then we have:
\begin{thm}[\cite{Laczkovich}, Lemma 2.2 and Theorem 1.3] \name{thm: Laczkovich3}
There is a constant $\kappa$, depending only on $d$, such that if
$\mc{C}$ is  a finite union of unit cubes in $\R^d$, then:
\begin{itemize}
\item[(i)]
$\left|(\partial \mc{C})^{(1)} \right| \leq \kappa |\partial \, \mc{C}|_{d-1} $
\item[(ii)]
 there are dyadic cubes $Q_1,
\ldots, Q_n$, such that $\mc{C} \in S(Q_1, \ldots, Q_n)$ and for each
$j$,
\eq{eq: laczko3}{
\# \, \{i: Q_i \mathrm{\ has \ sidelength \ } 2^j \} \leq \kappa \,
\frac{|\partial \, \mathcal{C} |_{d-1}}{2^{j(d-1)}}.
}
\end{itemize}
\end{thm}

\begin{cor}\name{cor: dynamical for BDD}
Suppose $\R^d$ acts on $X$ and $\mathcal{S}$ is a good section for the
action. Let $\mathcal{U}_1$ be a neighborhood of
identity in $\R^d$ satisfying (i) of \S \ref{subsection: sections},
and let $W$ be as in \equ{eq: defn V}.
Suppose there are
positive constants $a, C, \delta$ such that for all $x \in X$ and $T>1$,
\eq{eq: need to check, BDD}{
\Big|N_T(W,x) - a|B_T|  \Big| < C\, T^{d-1-\delta}.
}
Then for any $x \in X$, the net $Y_{\mathcal{S},x} $ as in \equ{eq: visit set} is BD to a
lattice.
\end{cor}

\begin{proof}
Let $W$ be as in \equ{eq: defn V}, where we assume with no loss of
generality that the diameter of $\mathcal{U}_1$ is smaller than 1. Let $b = |\mathcal{U}_1|$ and let $\lambda = a/b$.
Let $\mc{C}$ be a finite union of unit cubes, let
$\mc{C}^{(0)}$ denote the set of points of $\mc{C}$ whose
distance from $\partial \, \mc{C}$ is at least 1, and let
$\mc{C}^{(1)} $ denote the 1-neighborhood of $\mc{C}$. Then
$(\partial \, \mc{C})^{(1)} = \mc{C}^{(1)} \sm
\mc{C}^{(0)}$ and 
according to Theorem \ref{thm: Laczkovich3}(i),
\eq{eq: one estimate}{\left|\mc{C} ^{(1)} \right| - \left| 
\mc{C}^{(0)} \right|  = O\left(|\partial \, \mc{C} |_{d-1} \right)
}
 (where the
implicit constant depends only on the dimension $d$). 
Arguing as in the proof of Corollary \ref{cor: dynamical for BL}, 
we find that $$
\left|\#(\mc{C} \cap Y) b- a\int_{\mc{C}} \chi_W(\mathbf{t}.x) \,
d\mathbf{t} \right| \leq a\left(\left|\mc{C}^{(1)} \right| - \left| 
\mc{C}^{(0)} \right|\right), 
$$
and hence, in light of \equ{eq: one estimate},
$$
\Disc_Y(\mathcal{C}, \lambda) = \left| \int_{\mc{C}}
  \chi_W(\mathbf{t}.x)d\mathbf{t} - \lambda |\mc{C}| \right| + O(|\partial \,
\mc{C}|_{d-1}). 
$$
Thus, in light of condition (2) of Theorem \ref{thm:
  Laczkovich2}, it suffices to show that 
$$
 \left| \int_{\mc{C}}
  \chi_W(\mathbf{t}.x) \, d\mathbf{t} - \lambda |\mc{C}| \right| =  O(|\partial \,
\mc{C}|_{d-1}). 
$$
 Let $Q_1, \ldots, Q_n$ be a finite collection of cubes as in Theorem \ref{thm:
  Laczkovich3}(ii), so that we may write $\mc{C}$ using finitely many
operations of finite disjoint union and proper difference starting from the sets $Q_i$. Let $T_i$
be the sidelength of $Q_i$. Then for some choices of $\vre_i = \pm 1,
\, i=1, \ldots, n$, 
\[
\begin{split}
 & \left| \int_{\mc{C}}
  \chi_W(\mathbf{t}.x)  \, d\mathbf{t}  - \lambda |\mc{C}| \right| 
\\   =&  \left|
  \sum_{i=1}^n \vre_i \left(
\int_{Q_i}
  \chi_W(\mathbf{t}.x)d\mathbf{t}    - \lambda \, 
    |Q_i|  \right) 
\right| \\ 
\leq & \sum_{i=1}^n \left| \int_{Q_i}
  \chi_W(\mathbf{t}.x)d\mathbf{t}   -  \lambda |Q_i|\right |  
 \stackrel{\equ{eq: need to check, BDD}}{=} 
O \left(\sum_{i=1}^n
  T_i^{d-1-\delta} \right) \\
\stackrel{\equ{eq: laczko3}}{=} & O\left (\sum_{j}
\frac{(2^j)^{d-1-\delta}}{2^{j(d-1)}} \, |\partial \, \mc{C}|_{d-1} \right) =  O(|\partial \, \mc{C}|_{d-1}),
\end{split}
\]
as required. 
\end{proof}

\section{Minkowski dimension and approximation}\name{section:
  prelims}
Let $A \subset \R^k$ be bounded and let $r>0$. We denote by $N(A,r)$
the minimal number of balls of radius $r$ needed to cover $A$, and
$$
\dim_M A  \df \limsup_{r \to 0} \frac{\log N(A,r)}{-\log r}.
$$
Equivalently (see e.g. \cite[Chap. 3]{Falconer}), for $r>0$ let
$\mathcal{B}$ be the collection of boxes $[a_1, a_1 +r] \times \cdots
\times [a_k, a_k+r]$ where the $a_i$ are integer multiples of $r$, and
let $S(A, r)$ denote the number of elements of $\mathcal{B}$ which
intersect $A$. Then
$$
\dim_M A = \limsup_{r \to 0} \frac{\log S(A,r)}{-\log r}.
$$
It is clear from the definition that if $F$ is a bilipschitz map, then $\dim_M \, A =
\dim_M \, F(A)$. 

From Theorem \ref{thm: equidistribution for BL} we derive:
\begin{cor}\name{cor: for main, BL}
Let $\mathbf{v}_1, \ldots, \mathbf{v}_d \in \R^k$ be such that
$\spa \left(\mathbf{v}_1,\ldots ,\mathbf{v}_d\right)$ is Diophantine,
and suppose
$U$ is a closed set in $\T^k$, such that
$\dim_M \, \partial U < k$. Then
there are constants $C$ and $\delta>0$ such that for any
${\bf x}\in\R^k$ and any $T>1$,
\begin{equation*}
\Big|N_T(U,{\bf x})-|U||B_T|\Big|\leq C\, T^{d-\delta}.
\end{equation*}
\end{cor}
\begin{proof}[Proof (assuming Theorem
  \ref{thm: equidistribution for BL})]
Let $K$ be a positive integer and for each ${\bf m}\in\Z^k$ let
\[C(\mathbf{m})= \left[\frac{m_1}{K},\frac{m_1+1}{K}\right]
\times\cdots\times \left[\frac{m_k}{K},\frac{m_k+1}{K}\right]. \]
Define $A_1,A_2\subset\R^k$ by
\begin{align*}
A_1=\bigcup_{\substack{{\bf m}\in\Z^k\\ C({\bf m})\subset U}}C({\bf
  m})\quad\text{and}\quad A_2=\bigcup_{\substack{{\bf m}\in\Z^k\\
    C({\bf m})\cap U\not=\emptyset}}C({\bf m}).
\end{align*}
Clearly $N_T(A_1,{\bf x})\le N_T(U,{\bf x})\le N_T(A_2,{\bf x}),$ so that
\begin{align}\label{NTestimate1}
\Big|N_T(U,{\bf x})-|U||B_T|\Big|\le\max_{i=1,2}
\Big|N_T(A_i,\mathbf{x})-|U||B_T| \Big|.
\end{align}
Now by the triangle inequality
\begin{align}\label{NTestimate2}
\Big|N_T(A_1,{\bf x})-|U||B_T|\Big|\le \Big|N_T(A_1,{\bf
  x})-|A_1||B_T|\Big|+|B_T|\Big||A_1|-|U|\Big|.
\end{align}
The number of ${\bf m}\in\Z^k$ with $C({\bf m})\subset U$ is bounded
above by a constant times $M^k$, so applying Theorem \ref{thm: equidistribution for BL} to
each of the aligned boxes $C(\mathbf{m})$ gives
\[\Big|N_T(A_1,{\bf x})-|A_1||B_T|\Big|\leq c_1 T^{d-\delta_0}K^k,\]
where $c_1$ and $\delta_0$ are positive constants that are independent of
$K$. Now our
hypothesis on the dimension of the boundary
guarantees that there is an $\vre >0$ such that the number of
${\bf m}\in\Z^k$ for which $C({\bf m})$ intersects $\partial U$ is
bounded above by a constant times $K^{k-\vre}$. Each of these
boxes has volume $K^{-k}$ and thus we have that
\begin{align*}
|B_T|\Big||A_1|-|U|\Big| \leq c_2 \, |B_T|\frac{K^{k-\vre}}{K^k} \leq
c_3 \, T^dK^{-\vre},
\end{align*}
with $0<c_2<c_3$ independent of $K$.
Now we return to (\ref{NTestimate2}) and set $K=\lfloor
T^{\delta_0/(k+\vre)} \rfloor$ to obtain the bound
\[\Big|N_T(A_1,{\bf x})-|U||B_T|\Big|~\leq c_1
T^{d-\delta_0}K^k+c_3 T^dK^{-\vre} \leq (c_1+c_3)
T^{d-\delta_0 \vre/(k+\vre)}. \]
Setting $\displaystyle{C=c_1+c_3, \, \delta = \frac{\delta_0 \vre}{k+\vre}}$ and applying the same
argument to $A_2$ finishes the proof via (\ref{NTestimate1}).
\end{proof}

We now give a similar argument for bounded displacement.

\begin{cor}\name{cor: for main, BDD}
Suppose $d>(k+1)/2$ and $\mathcal{T} = (\mathbf{v}_1, \ldots, \mathbf{v}_k)$ is a basis of $\R^k$
satisfying the conditions of Theorem
\ref{thm: equidistribution for BDD}. 
Let
$\mathcal{S}$ be a good section lying in a translate of $\spa
(\mathbf{v}_{d+1},\ldots ,\mathbf{v}_k)$, which is closed in this affine
subspace, and satisfies
$\dim_M \partial \mathcal{S}=k-d-1$. Then we can choose
$\mathcal{U}_1$ satisfying (i) of \S \ref{subsection: sections} so
that, for the set $W$ defined as in \equ{eq: defn V}, there are
constants $C$ and $\delta>0$ such that for any ${\bf x}\in\R^k$ and
any $T>1$,
\begin{equation*}
\Big|N_T(W,{\bf x})-|W||B_T|\Big|\leq C\, T^{d-1-\delta}.
\end{equation*}
\end{cor}
\begin{proof}[Proof (assuming Theorem \ref{thm: equidistribution for BDD})]
Much of this proof is analogous to the previous one, so to simplify
the exposition we omit some of the notational details. We begin by
covering the set $\mathcal{S}$ by $(k-d)$-dimensional boxes which are
translates of aligned boxes in $\spa
(\mathbf{v}_{d+1},\ldots ,\mathbf{v}_k)$ of sidelength $\eta=1/K,~K\ge
1$. As before we construct disjoint unions 
$A_1,A_2$ of such boxes with the property that $A_1\subset
\mathcal{S}\subset A_2$, and we have that
\[\Big|N_T(W,{\bf x})-|W||B_T|\Big|\le\max_{i=1,2}
\Big|N_T(A_i',\mathbf{x})-|W||B_T| \Big|,\]
with
\[A_i' \df \{u.x: u \in \mathcal{U}_1,\, x \in A_i \}.\]
We choose $\mathcal{U}_1$ to be any parallelotope in $\R^d$ which
satisfies (i) of \S \ref{subsection: sections}, and which has sides
parallel to $\mathbf{v}_1, \ldots, \mathbf{v}_d$. This is clearly
possible since we can always replace our original choice of this set
by any sub-neighborhood of the origin. With this choice of
$\mathcal{U}_1$ our sets $A_i'$ are unions of parallelotopes aligned
with $\mathcal{T}$, with a uniform bound on their sidelengths. That
is, parallelotopes to which Theorem \ref{thm:
  equidistribution for BDD} applies. The number of parallelotopes in
$A_1'$ is bounded above by a constant times $K^{k-d}$, so Theorem
\ref{thm: equidistribution for BDD} tells us that for any $\delta_0 >0
$ there is a $c_1>0$  (independent of $K$) for which 
\[\Big|N_T(A_1',{\bf x})-|A_1'||B_T|\Big|\leq c_1  T^{\delta_0}K^{k-d}.\]
Our hypothesis that $\dim_M \partial \mathcal{S}=k-d-1$ leads to the inequality
\[|B_T|\Big||A_1'|-|W|\Big| \leq c_2 \, |B_T|\frac{K^{k-d-1}}{K^{k-d}} \leq
\frac{c_3T^d}{K},\]
and using the triangle inequality as in (\ref{NTestimate2}) we have that
\[\Big|N_T(A_1',{\bf x})-|W||B_T|\Big|\le c_1 T^{\delta_0}K^{k-d}+\frac{c_3T^d}{K}.\]
Now using the hypothesis that $d>(k+1)/2$, we may assume that
$\delta_0$ has been chosen small enough so that there is a
$\delta>\delta_0$ with $(1+\delta)(k-d)<(d-1-2\delta)$. Then setting
$K=\lfloor T^{1+\delta}\rfloor$ we have that 
\[\Big|N_T(A_1',{\bf x})-|W||B_T|\Big|\le c_4T^{d-1-\delta}.\]
Since the same analysis holds for $A_2'$, the proof is complete. 
\end{proof}

\ignore{
The following is a standard fact: \combarak{Find a reference or prove,
  or derive from more standard facts. perhaps this should not be a
  separate statement, but stuck in the proof of the theorem.}
\begin{prop}\name{prop: dim enlargement}
Suppose $\mathcal{S}$ is a linear section as in Theorem \ref{thm:
  main, BL}, suppose $\mathcal{U}_1$ is a closed ball around $0$ in $\R^d$,
satisfying (i) of \S \ref{subsection: sections}, and define $V$ via \equ{eq: defn V}.
Then $\dim_M \partial V < k$.
\end{prop}
}

\section{Trigonometric polynomials approximating  aligned
  parallelotopes}\name{section: trig}

The proofs of Theorems \ref{thm: equidistribution for BL} and
\ref{thm: equidistribution for BDD} proceed with two major steps. The
first step is to prove an Erd\H{o}s-Tur\'an type inequality for Birkhoff
integrals, and the second is to use Diophantine properties of the
acting subspace to produce a further estimate on the error terms
coming from the Erd\H{o}s-Tur\'an type inequality.  Our goal in this
section is to build up the necessary machinery to complete the first
step. 

Our approach to proving the Erd\H{o}s-Tur\'an type inequality
requires approximations of the indicator function of an aligned
  parallelotope by trigonometric polynomials which {\it majorize} and
{\it minorize} it. To obtain the quality of estimates that we need, we
require the trigonometric polynomials to be close to the indicator
function of the parallelotope in $L^{1}$-norm and to have suitably
fast decay in their Fourier coefficients. The following theorem is the
main result of this section, the Fourier analysis notation will be
explained shortly. 

\begin{thm}\name{thm:extremal}
  Suppose that $\mathcal{T}=({\bf t}_{1},...,{\bf t}_{k})$ is a basis for $\R^k$ and that $L$ is the linear isomorphism mapping ${\bf e}_{i}$ to ${\bf t}_{i}$. Suppose $U\subset \R^k$ is a parallelotope, aligned with $\mathcal{T}$, given by
  $U=LB$ for a box \[B=\displaystyle\prod_{\ell=1}^{k}[-b_{\ell},b_{\ell}]\] such that $\pi|_U$ is injective. Let $\chi^{\T}_U: \T^k \to \R$ denote the
  indicator function of $\pi(U)$. 
Then for each $M\in\N$ there are trigonometric polynomials
$\varphi_{U}(\bx)$ and $\psi_{U}(\bx)$ whose Fourier coefficients are
supported in $\{ \bm\in\Z^{k}:\|L^{t}\bm\|\leq M \}$, where $L^{t}$
denotes the transpose of $L$, and 
\eq{eq:majMin}{
	\varphi_{U}(\bx)\leq \chi^{\T}_{U}(\bx)\leq\psi_{U}(\bx)
}
for each $\bx\in\T^{k}$. Moreover, there exists a constant $C>0$, depending only on $k$,  such that
\eq{eq:zeroFourier}{
	\max\left\lbrace |U|-\hat{\varphi}_{U}({\bf 0}), \hat{\psi}_{U}({\bf 0})-|U| \right\rbrace\leq \dfrac{Cb^{k-1}|\det L|}{M},
}
and the Fourier coefficients of $\varphi_{U}(\bx)$ and $\psi_{U}(\bx)$ satisfy
\eq{eq:FourierCoeffs}{
\max\left\lbrace \hat{\varphi}_{U}(\bm),\hat{\psi}_{U}(\bm) \right\rbrace\leq k2^{k+1}(1+2b)^{k}|\det L| r_{\mathcal{T}} (\bm)
}
for all nonzero $\bm\in\Z^k,$ where $r_{\mathcal{T}} (\bm)$ is defined by \equ{eq: defn r general} and $b=\displaystyle\max_{\ell}\{ b_{\ell} \}$.
\end{thm}

\indent We note that some form of this result is alluded to in \cite[Proof of Theorem 5.25]{Harman1998}, and since we could not find a suitable reference we will give the full details here. There are, however, known constructions which handle the case when $U$ is rectangular \cite{BMV,MR947641,DT,Harman1993}.  Our proof requires the well known construction of Selberg regarding extremal approximations of the indicator functions of intervals by integrable functions with compactly supported Fourier transforms, which we will recall below. To move to several variables we bootstrap from the single variable theory using another construction due to Selberg, who never published his results. A similar construction can be found in \cite{MR947641,Harman1993}. \newline
\indent Let $e(x) \df \exp(2\pi i x)$. We will use the same
notation for the Fourier transform of a function $F \in L^1(\R^{N})$ and
for a function $f:\R^{N} \to \R$ which is periodic with respect to $\Z^N$. That is
\[\hat{F}({\bf t}) \df \int_{\R^{N}} F({\bf x})e(-{\bf t\cdot x})~d{\bf x}, \ {\bf t} \in \R^{N}; \ \
\hat{f}({\bf m}) \df \int_{[0,1)^{N}} f(\boldsymbol{\theta})e(-{\bf m\cdot}\boldsymbol{\theta})~d\boldsymbol{\theta}, \ {\bf m} \in \Z^{N}.\]
The reader should have no difficulty distinguishing these two uses. \newline

If $I\subset \R$ is an interval, let $ \chi(t)=\chi_{I}(t)$ be its indicator function. The following lemma is due to Selberg.
\begin{lem}\name{lem:SelbergLemma}
For each positive integer $M$ there exist integrable functions $C_{I},c_{I}:\R\rightarrow\R$ such that
  \begin{enumerate}
   \item $c_{I}(t)\leq \chi(t)\leq C_{I}(t)$ for each $t\in\R$;
   \item $\hat{C}_{I}(\xi)=\hat{c}_{I}(\xi)=0$ whenever $|\xi|\ge M$,  
   \item \begin{equation} \label{eq:norms}
	\|C_{I}-\chi\|_{L^{1}(\R)}=\|\chi-c_{I}\|_{L^{1}(\R)}=\frac1M,\text{ and}
    \end{equation}
   \item $$ 
	 \max\left\{\left|\hat{C}_{I}(\xi)\right|,
         \left|\hat{c}_{I}(\xi)
         \right|
         \right\}\leq \min\lb 1+|I| , \dfrac{2}{|\xi|}\rb.
	$$ 
	for each $\xi\in\R$.
  
  \end{enumerate}
\end{lem}
\begin{proof}
	We will only verify the estimates on the Fourier coefficients appearing in (4) above. The other properties are well known and can be found in \cite{Vaaler1985} or in \cite{SelbergCollectedWorks}.
						From (\ref{eq:norms}) we have
					$$ 
							    \displaystyle\sup_{\xi\in\R}\left|\hat{C}_{I}(\xi)-\hat{\chi}
                                                            (\xi)\right|\leq \|C_{I}-\chi\|_{L^{1}(\R)}=\frac1M.
					$$ 
						In particular for any fixed $\xi$ we have
						      \begin{equation} \label{eq:ineq1}
						\left	    |\hat{C}_{I}(\xi)\right| \leq \left|\hat{\chi}(\xi)\right| + \frac1M.
						      \end{equation}
					For any $1<|\xi|<M$ we have 
					$ 
							   \left |\hat{\chi}(\xi)\right|=|\sin(\pi\xi|I|)/\pi\xi|\leq |\pi\xi|^{-1},
					$ 
hence
$$
	    |\hat{\chi}(\xi)| + \frac1M < |\pi\xi|^{-1}+ |\xi|^{-1}
            < \frac{2}{|\xi|},
$$

	therefore
					$$
							    |\hat{C}_{I}(\xi)|\leq
                                                            \dfrac{2}{|\xi|}~\text{ for }~1<|\xi|<M.
					$$
						Recall that
                                                $\hat{C}_{I}(\xi)=0$
                                                if $|\xi|\geq M$ so it
                                                remains to show
                                                $|\hat{C}_{I}(\xi)|\leq
                                                1+|I|$ when $|\xi|<1$. But
                                                by (\ref{eq:ineq1}) we
                                                have
$$
| \hat{C}_{I}(\xi)| \leq \sup_{|\xi| < 1}
\left|\frac{\sin(\pi\xi|I|)}{\pi \xi} \right| + \frac1M \leq |I|+1.
$$
This concludes the proof for $\hat{C}_I$, and the proof for
$\hat{c}_I$ is nearly identical.
					\end{proof}
\subsection{Majorizing and minorizing a rectangle in $\R^k$}\name{sec: majorizing}
			 From here on out we will use the notation $C_{i}(x)=C_{[-b_{i},b_{i}]}(x)$.   For any $M\in\N$ the indicator
                            function $\chi_{B}$ of
                            $B\subset \R^k$ is clearly
                            majorized by the function \newline
                            
				      \begin{equation} \label{eq:maj}
					  G_{B}({\bf x}) \df \displaystyle\prod_{j=1}^{k}C_{j}(x_{j}).
				      \end{equation}
			
			    Minorizing $\chi_{B}$ requires a little more effort. For $i=1,2,\ldots ,k$ define
	$$ 
				      L_{B}({\bf
                                        x};i) \df c_{i}(x_{i})\underset{j\neq
                                        i}{\displaystyle\prod_{j=1}^{k}}C_{j}(x_{j}),
		$$ 
and then set
			$$ 
					g_{B}({\bf x}) \df  -(k-1)G_{B}({\bf x})+\dsum_{i=1}^k L_{B}({\bf x};i).
			$$ 
			
				We claim that
\begin{equation}\label{eq:g_Mminorant}
g_{B}(\bx)\leq \chi_{B}(\bx)~\text{ for every }\bx\in\R^k.
\end{equation}
To establish this we use the following  elementary inequality, which
can be proved by induction on $k$:
\eq{eq: lem:prodlem}{
%
\mathrm{For \  any} \ \beta_{1} \geq 1, \ldots ,\beta_{k} \geq 1, \ \ \
\dsum_{i=1}^{k} \, \underset{j\neq
  i}{\displaystyle\prod_{j=1}^{k}}\beta_{j} \leq 1+ (k-1) \displaystyle\prod_{j=1}^{k}\beta_{j}.
}

To verify the inequality (\ref{eq:g_Mminorant}), first suppose
that $\bx\not\in B$. Then there is an $1\le i\le k$ with
$|x_i|>b_{i}$. Since $L_{B}(\bx;i)\leq 0$ and $L_{B}(\bx;j)\leq
G_{B}(\bx)$ for all $j\not= i$, we have that
\[\dsum_{i=1}^k L_{B}(\bx;j) \leq (k-1)G_{B}(\bx),\]
which implies
$     g_{B}(\bx)\leq 0$.
On the other hand if $\bx\in B$ then we have that     
				  			\begin{equation*}
				  					c_{j}(x_{j})\leq 1 \leq C_{j}(x_{j}).
							\end{equation*}
Then by \equ{eq: lem:prodlem} we have that	
							\begin{eqnarray*}
							 	\dsum_{i=1}^{k}L_{B}(\bx;i)&
                                                                \leq
                                                                &\dsum_{i=1}^{k}\underset{j\neq
                                                                  i}{\displaystyle\prod_{j=1}^{k}}C_{j}(x_{j})\le
                                                                1+
                                                                (k-1)G_{B}(\bx),
							\end{eqnarray*}
and this together with the definition of $g_B$ establishes
(\ref{eq:g_Mminorant}).


\subsection{Proof of Theorem \ref{thm:extremal} }

Define
				   \begin{equation*}
					  \mathscr{G}_{U}(\bx) \df G_{B}\circ L^{-1}(\bx)
\ \mathrm{			    and} \
					  \mathscr{F}_{U}(\bx) \df g_{B}\circ L^{-1}(\bx).
				   \end{equation*}				
			    The results of \S \ref{sec: majorizing} show that
				    \begin{equation*}
					  \mathscr{F}_{U}(\bx)\leq\chi_{U}(\bx)\leq\mathscr{G}_{U}(\bx)~\text{
                                            for all }~\bx\in\R^k.
				    \end{equation*}
For the majorants and minorants of $\chi^\T_U$ define
				    \begin{equation*}
					  \varphi_{U}(\bx) \df \dsum_{\bm\in\Z^k}\mathscr{F}_{U}(\bx+\bm)	
			  \ \ \ \mathrm{and} \ \
				  \psi_{U}(\bx) \df \dsum_{\bm\in\Z^k}\mathscr{G}_{U}(\bx+\bm).
				    \end{equation*}
These functions are $\Z^k$ invariant, so we can view them as functions
on $\T^k$, and since $\pi|_U$ is injective we have
				    \begin{equation}\label{eq:chiTmaj-min}
					  \varphi_{U}(\bx)\leq \chi^\T_U(\bx) \leq \psi_{U}(\bx)~\text{ for all }~\bx\in\T^k.
				    \end{equation}
To determine the Fourier transform of $\msG_U$ and $\msF_U$, observe
that if $f:\R^k\rightarrow\C$ is an integrable function then $f\circ
L^{-1}$ is also integrable and
\begin{equation}\label{eq:Tfourcoeffs}
						    \widehat{f\circ L^{-1}}(\bxi)=
                                                                              |\det
                                                                              L|\hat{f}(L^{t}\bxi).
						    \end{equation}
	Since $\hat{G}_{B}(\bxi)=0$ and $\hat{g}_{B}(\bxi)=0$ when
        $\|\bxi\|\geq M$, both $\hat{\mathscr{F}}_{U}$ and
        $\hat{\mathscr{G}}_{U}$ are supported on $\{\bxi \in \R^k
        : \|L^{t} \bxi \| \leq M\}$. Thus, by the
        Poisson summation formula and a classical theorem of P\'olya and Plancherel \cite{PP}, we have the following {\it pointwise} identities
				    \begin{equation} \label{eq:rho1}
					  \psi_{U}(\bx)=\dsum_{\substack{\bm\in\Z^k \\
                                            \| L^{t} \bm \| \leq M}}\hat{\mathscr{G}}_{U}(\bm)e(\bm\cdot\bx)
				    \end{equation}		
		      and
				    \begin{equation} \label{eq:eta1}
					  \varphi_{U}(\bx)=\dsum_{\substack{\bm\in\Z^k \\
                                            \|L^{t} \bm \| \leq M }}\hat{\mathscr{F}}_{U}(\bm)e(\bm\cdot\bx).
				    \end{equation}
We will
need the following formulas for the Fourier coefficients of $\psi_U$
and $\varphi_U$:
					    \eq{eq: first formula}{
						  \hat{\psi}_{U}(\bm)=|\det
                                                 L|\displaystyle\prod_{i=1}^{k}\hat{C}_{i}(\bt_{i}\cdot\bm)
					    }
and
					    \eq{eq: second formula}{
						  \hat{\varphi}_{U}(\bm)=
                                                  |\det L|\left(-(k-1)\displaystyle\prod_{i=1}^{k}\hat{C}_{i}(\bt_{i}\cdot\bm)+\dsum_{i=1}^{k}\hat{c}_{i}(\bt_{i}\cdot
                                                    \bm)\underset{j\neq
                                                      i}{\displaystyle\prod_{j=1}^{k}}\hat{C}_{j}(\bt_{j}\cdot\bm)\right).
					    }
			To see \equ{eq: first formula}, first observe that 	$\hat{\psi}_{U}(\bm)=\hat{\mathscr{G}}_{U}(\bm)$ then 
                                                        by
                                                        (\ref{eq:Tfourcoeffs})
                                                        and basic
                                                        properties of
                                                        the Fourier
                                                        transform we
                                                        have that
								  \begin{eqnarray*}
									\hat{\mathscr{G}}_{U}(\bm)= 
                                                                        \widehat{G_{B}\circ
                                                                          L^{-1}}(\bm) &=& |\det L |\hat{G}_{B}(L^{t}\bm)\\
												  &=&
                                                                                                  |\det L |\displaystyle\prod_{i=1}^{k}\hat{C}_{i}(\bt_{i}\cdot\bm).
								  \end{eqnarray*}
The proof of \equ{eq: second formula} is similar.
				By (\ref{eq:norms}) we see
that  \begin{equation}\label{eq:C^0}
\hat{C}_{i}(0)=2b_{i}+\frac1M~\text{ and }~\hat{c}_{i}(0)=2b_{i}-\frac1M.
\end{equation}
			Now by using \equ{eq: first formula} and
                        \equ{eq: second formula}, together with (\ref{eq:C^0}) we find that
\[ 
								  \hat{\psi}_{U}({\bf 0})=|\det L |\displaystyle\prod_{i=1}^{k}\left(2b_{i}+M^{-1}\right)
\] 
						    and
							  \begin{eqnarray*}
								\hat{\varphi}_{U}({\bf
                                                                  0})
                                                                &=&
                                                                |\det L |\left(-(k-1)\displaystyle\prod_{i=1}^{k}(2b_{i}+M^{-1})+\dsum_{j=1}^{k}(2b_{j}-M^{-1})\underset{i\neq
                                                                    j}{\displaystyle\prod_{i=1}^{k}}(2b_{i}+M^{-1}) \right)      \\
										  &=& 2^{k}b_{1}\cdots b_{k}|\det  L| +|\det  L|O(b^{k-1}/M).
 							  \end{eqnarray*}
The bounds \equ{eq:zeroFourier} follow upon recalling that $|U|=2^{k}b_{1}\cdots b_{k}|\det L|$. For the other Fourier coefficients we use (4) from Lemma \ref{lem:SelbergLemma} to obtain the inequalities
							\begin{equation*}
							      |\hat{\psi}_{U}(\bm)|
                                                              = |\det L |    \displaystyle\prod_{i=1}^{k}\left|\hat{C}_{i}(\bt_{i}\cdot\bm)\right|
                                                         \le
                                                              2^{k}(1+2b)^{k}|\det L |
                                                              \displaystyle\prod_{i=1}^{k}\min\lb
                                                              1,
                                                              \dfrac{1}{|\bt_{i}\cdot\bm|}\rb
							\end{equation*}		
and
							\begin{eqnarray*}
							      |\hat{\varphi}_{M}(\bm)|
										    &\leq&
                                                                                    \left|\det L \right|\left(\left|(k-1)\displaystyle\prod_{i=1}^{k}\hat{C}_{i}(\bt_{i}\cdot\bm)\right|+\dsum_{j=1}^{k}\left|\hat{c}_{j}(\bt_{j}\cdot
                                                                                        \bm)\underset{i\neq
                                                                                          j}{\displaystyle\prod_{i=1}^{k}}\hat{C}_{i}(\bt_{i}\cdot\bm)\right|
                                                                                    \right)
                                                                                    \\
										    &\leq&
                                                                                    2^{k}(2k-1)(1+2b)^{k}|\det L |\displaystyle\prod_{i=1}^{k}\min\lb
                                                                                    1,
                                                                                    \dfrac{1}{|\bt_{i}\cdot\bm|}\rb.
							\end{eqnarray*}	
Combining these estimates with (\ref{eq:chiTmaj-min}),
(\ref{eq:rho1}), and (\ref{eq:eta1}) finishes our proof. \qed

\section{An Erd\H{o}s-Tur\'an type inequality for Birkhoff
  integrals}\name{section: ET}
From Theorem \ref{thm:extremal} we deduce:\newline

\begin{thm}\name{thm:NTestimate}
For any positive integer $k\geq 2$ there is a constant $C>0$, depending only on $k$, such that the
following holds. Suppose that $d < k$ is a positive integer and that $V\subset\R^k$
is a subspace of dimension $d$ spanned by $\{{\bf v}_1,\ldots ,{\bf v}_d\}.$ Let $\tilde{L} :
\R^k \to \R^k$ be
an affine isomorphism such that $\pi$ is injective on the parallelotope
$U=\tilde{L}B$ where $B$ and $b$ are as in Theorem \ref{thm:extremal}. Let $\mathcal{T}$ denote the basis $L(\E_i), i=1,
\ldots, k$, where $L$ is the linear part\footnote{We recall that the linear part of an affine isomorphism $\tilde{L}:\R^k\rightarrow\R^k$ is the unique linear transformation $L:\R^k\rightarrow\R^k$ such that $L\bx=\tilde{L}\bx-\tilde{L}{\bf 0}$ for each $\bx\in\R^k$.} of $\tilde{L}$. Then for any $M\in\N$
and ${\bf x}\in\R^k$ we have
\eq{eq:NTestimate}{
\Big|N_T(U,{\bf
  x})-|U||B_T|\Big|\leq C (1+2b)^{k}|\det L| \left( \frac{|B_T|}{M}+\sum_{\substack{{\bf m}\in
    \Z^k \sm \{0\} \\ \|L^{t} {\bf m}\| \leq M}}r_{\mathcal{T}}({\bf
  m})\left|\int_{B_T}e({\bf m}\cdot {\bf s})~d{\bf s}\right| \right).
}
\end{thm}
\begin{proof}
If $\tilde{L}(\by) = L(\by)+\by_0$, we may replace $\bx$ with
$\bx-\by_0$ to assume that $\tilde{L}= L$, so that Theorem
\ref{thm:extremal} applies.
For $M\ge 1$ we have from Theorem \ref{thm:extremal}
\eq{eq:majorantPair}{
\chi_{U}^{\T}({\bf x}) - |U|\leq \psi_{U}({\bf
  x})-|U|\leq\dfrac{C'b^{k-1}|\det L|}{M}+\sum_{\substack{{\bf m}\in\Z^k \sm \{0\} \\
    \|L^{t} {\bf m}\| \leq M}}\hat{\psi}_{U}({\bf m})e({\bf m\cdot
  x}),
}
for some constant $C'$ which depends only on $k$.
By integrating both sides of
\equ{eq:majorantPair} over $B_T -\bx$ we find  that
\begin{align*}
N_T(U,{\bf x})-|U||B_T|&\le
\left|\frac{C' b^{k-1}|\det L| \, |B_T|}{M} +\int_{B_T}\sum_
{\substack{{\bf m}\in
    \Z^k \sm \{0\} \\ \|L^{t} {\bf m}\| \leq M}}
\hat{\psi}_U(\bm) e({\bf
    m}\cdot {\bf s})d{\bf s}\right|\\
&\leq  \frac{C'b^{k-1}|\det L| \, |B_T|}{M}+\sum_{\substack{{\bf m}\in
    \Z^k \sm \{0\} \\ \|L^{t} {\bf m}\| \leq M}}
\left|\hat{\psi}_U({\bf
    m})\right|\cdot\left|\int_{B_T}e({\bf m}\cdot {\bf s})~d{\bf
    s}\right| \\ 
& \stackrel{\equ{eq:FourierCoeffs}}{\leq} C (1+2b)^{k}|\det L| \left ( \frac{|B_T|}{M} + 
\sum_{\substack{{\bf m}\in
    \Z^k \sm \{0\} \\ \|L^{t} {\bf m}\| \leq M}} r_{\mathcal{T}}({\bf
  m})
\cdot\left|\int_{B_T}e({\bf m}\cdot {\bf s})~d{\bf
    s}
\right| \right),
\end{align*}
where $C \df \max (C', k2^{k+2})$. 
For a lower bound on
$N_T(U,{\bf x})-|U||B_T|$ we use $\varphi_{U}({\bf x})\leq \chi_{U}^{\T}({\bf x})$ in a
similar way.
\end{proof}

Specializing to aligned boxes we obtain a generalization of the
Erd\H{o}s-Tur\'{a}n inequality.
\begin{cor}\name{lem: erdos turan}
Let the notation be as in Theorem \ref{thm: equidistribution for
  BL}. Suppose $U \subset \R^k$ is an aligned box. Then there is a
positive constant $C$ (depending only on $k$) such that for any
$M\in\N$ and
${\bf x}\in\R^k$ we have that
\eq{eq: erdos turan}{
\Big|N_T(U,{\bf x})-|U||B_T|\Big|\le
C\left(\frac{|B_T|}{M}+\sum_{\substack{\mathbf{m}
      \in \Z^k \sm \{0\} \\ \|\mathbf{m} \| \leq M}}r({\bf m})\left|\int_{B_T}e({\bf
      m}\cdot {\bf t})~d{\bf t}\right|\right),
}
where
\eq{eq: defn r}{
r(\mathbf{m}) \df \prod_{i=1}^k \min \left(1, \frac{1}{|m_i|}
\right).
}
\end{cor}

\begin{proof}
In this case $L$ is the identity matrix, so that $B=U$ and $b\leq 1/2$.
\end{proof}

We will need the following estimate for the integrals appearing on the
right-hand-side of \equ{eq:NTestimate}:
\begin{prop}\name{prop: another bound}
There is a constant $C$ (depending only on $d, k$ and the choice of
Lebesgue measure on $V$) such that
\eq{eq: one more bound}{
\left|\int_{B_T}e({\bf m}\cdot {\bf s})~d{\bf s}\right| \leq C
\prod_{i=1}^d\frac{1}{|{\bf m}\cdot {\bf v}_i|}.
}
\end{prop}

\begin{proof}
For a constant $C_1$ depending on the choice of Lebesgue measure on $V$,
we have:
\begin{align}
\left|\int_{B_T}e({\bf m}\cdot {\bf s})~d{\bf s}\right| 
&=C_1 \prod_{i=1}^d\left|\int_{-T}^Te(({\bf m}\cdot {\bf v}_i)s_i)~ds_i\right|\nonumber\\
&=C_1 \prod_{i=1}^d\frac{|\sin(2\pi({\bf m}\cdot {\bf v}_i)T)|}{\pi|{\bf m}\cdot {\bf v}_i|}\nonumber\\
&\leq \frac{C_1}{\pi^d}\prod_{i=1}^d\frac{1}{|{\bf m}\cdot {\bf v}_i|}.\nonumber
\end{align}
\end{proof}

\begin{proof}[Proof of Theorem \ref{thm: equidistribution for BL}]
Let ${\bf v}$ be a Diophantine vector in the subspace spanned by ${\bf
  v}_1,\ldots , {\bf v}_d$, and write
${\bf v}=\sum_{i=1}^dx_i{\bf v_i}.$
Fix $c, s$ as in \equ{eq: defn Diophantine vector}, and let $s'>s$.
If ${\bf m}\in\Z^k$ satisfies
\[\max_{1\le i\le d}|{\bf m}\cdot {\bf v}_i|\le \|{\bf m}\|^{-s'}\]
then, for all but finitely many $\mathbf{m}$,
\[|{\bf m}\cdot {\bf v}|\le \left(\sum_{i=1}^d|x_i|\right)\|{\bf
  m}\|^{-s'}
\leq c \|\mathbf{m}\|^s
.\]
Thus for some $c_1>0$ we have
\eq{eq: thus}{
\max_{1\le i\le d}|{\bf m}\cdot {\bf v}_i|\ge c_1 \|{\bf
  m}\|^{-s'}~\text{ for all }~{\bf m}\in\Z^k.}
We will apply
Corollary \ref{lem: erdos turan} with
\eq{eq: floor}{
M=\lfloor T^{\delta} \rfloor, \ \ \mathrm{where} \
\delta = \frac{1}{d+s'+1}.
}

Assume that the maximum in \equ{eq: thus} is attained for $i=1$.
It follows 
that for some $c_4, c_3, c_2>0$,
\begin{align}
\left|\int_{B_T}e({\bf m}\cdot {\bf t})~d{\bf
    t}\right|& =c_2 \left|\int_{[-T,T]^d}e\left({\bf m}\cdot \left(\sum
      t_i {\bf v}_i \right) \right)~d{\bf t}\right|\nonumber\\
&= c_2 \, \prod_{i=1}^d\left|\int_{-T}^Te(({\bf m}\cdot {\bf
    v}_i)t_i)~dt_i\right| \nonumber\\
&\leq c_2 \, \frac{|\sin(2\pi({\bf m}\cdot {\bf v}_1)T)|}{\pi|{\bf m}\cdot
  {\bf v}_1|} \, \left( 2T \right)^{d-1} \, \nonumber \\
&\leq c_3 \, \frac{T^{d-1}}{|\mathbf{m} \cdot \mathbf{v}_1|}
\stackrel{\equ{eq: thus}, \equ{eq: floor}}{\leq} c_4 \, T^{d-1+s'\delta}. \nonumber
\end{align}
Therefore
\begin{align*}
\sum_{0<\|{\bf m}\| \le M}r({\bf
    m})\left|\int_{B_T}e({\bf m}\cdot {\bf t})~d{\bf t}\right|&\leq
c_4 \,
T^{d-1+ s'\delta}\sum_{0<\|{\bf m}\|\le M}r({\bf m}) 
\\
&\le c_4 \, M^d \, T^{d-1+s'\delta} 
\stackrel{\equ{eq: floor}}{\leq} c_5 
\, T^{d-\delta}.
\end{align*}
It is clear that the constants $c_5, \delta$ do not depend
on $U$ or ${\bf x}$. Thus the theorem follows from Corollary \ref{lem:
  erdos turan}.
\end{proof}

\begin{remark}
The proof shows that if $V$ is Diophantine with corresponding constant
$s$, then $\delta$ can be taken to be any number smaller than $\frac{d+1}{d+s+1}.$
\end{remark}

\begin{proof}[Proof of Theorem
\ref{thm: equidistribution for BDD}]
For a fixed $\delta$, let $\vre = \delta/d$ and let $C_1, C_2, C_3$ be the
constants $C$ appearing in \equ{eq: for strongly Diophantine},
\equ{eq:NTestimate} and
\equ{eq: one more bound} respectively. 
Let $L: \R^k \to \R^k$ be the linear isomorphism mapping
$\mathbf{e}_i$ to $\mathbf{v}_i$, $i=1, \ldots, k$. Any $U$ which is a
parallelotope aligned with $\mathcal{T}$ is of the form $U = \pi \circ
\til{L}(B)$, where $\til L$ is an affine isomorphism whose
linear part is $L$ and $B=\prod [-b_{i},b_{i}]$ .\newline
Let $\|\cdot\|_{2}$ be the Euclidean norm on $\R^k$. If the largest side length of $U$ is $\eta$, then 
\[
	\eta=2\displaystyle\max_{i} b_{i}\|{\bf v}_{i}\|_{2}.
\]
In particular $\eta\geq 2 b\min\|{\bf v}_{i}\|_{2}$ where $b=\max b_{i}$.
There is a constant
$\lambda$, which depends only on the $\mathbf{v}_1, \ldots, \mathbf{v}_k$, such that 
 \[\{\bm\in\Z^{k}: \|L^{t} {\bf m}\| \leq M\}\subset
 \{\bm\in\Z^{k}: \|{\bf m}\| \leq \lambda M\}.\] 
Applying Proposition \ref{prop: another bound}, we find that for any
$M>0$, the right
hand side of \equ{eq:NTestimate} is bounded above by 
\[C_2 |\det L| (1+\eta/\min\|{\bf v}_{i}\|_{2})^{k}\left( \frac{|B_T|}{M}+C_3 \sum_{\substack{{\bf m}\in
    \Z^k \sm \{0\} \\ \|{\bf m}\| \leq \lambda M}}r_{\mathcal{T}}({\bf
  m})\prod_{i=1}^d\frac{1}{|{\bf m}\cdot {\bf v}_i|}\right).\]
Now taking $M=\left\lfloor T^d \right\rfloor$, and using our strongly
Diophantine hypothesis, gives the required bound, with 
\[ 
	C=C_1 C_2 C_3  |\det L|\lambda^{\delta/d}\max\{1,1/\min\|{\bf v}_{i}\|_{2}\}^{k}.
\] 
\end{proof}



\section{Diophantine approximation to subspaces}
\name{section: reduction}
The main result of this section shows that the Diophantine properties
stated in the introduction hold almost surely. More precisely,
properties of $d$-tuples of vectors in $\R^k$ hold almost everywhere
with respect to Lebesgue measure on $\times_1^d \, \R^k \cong \R^{kd}$, and
properties of vector spaces hold almost everywhere with respect to the
smooth measure class on the Grassmannian variety.

The fact that almost every vector is Diophantine follows from the Borel-Cantelli Lemma --- or see \cite{DV} for a
stronger statement. For the extension to strongly
Diophantine vectors, we employ ideas of Schmidt
\cite{Schmidt1964}:
\begin{prop} \name{prop: required bound}
Almost every $\mathbf{v}_1, \ldots, \mathbf{v}_d$ is strongly Diophantine with respect to any basis $\mathcal{T} = (\mathbf{t}_1, \ldots, \mathbf{t}_k)$ for
$\R^k$ having the property that for each $i\in\{d+1,\ldots , k\}$, there is a $j$ for which $\mathbf{t}_i$ is a multiple of $\E_j$.
\end{prop}

\begin{proof}
Fix $\vre>0,$ let $R_1,\ldots , R_d$ be cubes in $\R^k$
of sidelength $1$, and for each $1\le i\le d$ and ${\bf
  m}\in\Z^k\sm\{{\bf 0}\}$ let
\[I_{R_{i}}({\bf m}) \df \int_{R_i}\frac{d{\bf v}}{|{\bf m}\cdot {\bf
    v}|(-\log \min(1/2,|{\bf m}\cdot {\bf v}|))^{1+\vre}}.\]
We estimate this integral by using the change of variables $\mathbf{u}
= \mathbf{u}(\mathbf{v})$, where
\[u_i={\bf m}\cdot {\bf v},\quad u_j=v_j \text{ for }1\le j\le k,~j\not= i.\]
The Jacobian determinant of this transformation is $1/m_i$. If we write $R_i'$ for
the image of $R_i$ in the
${\bf u}$ coordinate system then it is clear that for $j\not= i$ the
$u_j$ coordinates of two points in $R_i'$ cannot differ by more than
$1$. Using this fact we have
\begin{align*}
I_{R_i}({\bf m})\le \frac{2}{m_i}\int_0^{1/2}\frac{du_i}{u_i|\log
  u_i|^{1+\vre}}+\frac{1}{m_i(\log
  2)^{1+\vre}}\int_{1/2}^{1/2+m_i}\frac{du_i}{u_i}\leq
c_1 \frac{\log(m_i)}{m_i},
\end{align*}
where $c_1$ depends only on $\vre$.
Thus we have
\begin{align*}
\sum_{m_1=1}^\infty\cdots\sum_{m_d=1}^\infty \frac{I_{R_1}({\bf
    m})\cdots I_{R_d}({\bf m})}{(\log m_1)^{2+\vre}\cdots (\log
  m_d)^{2+\vre}}  \leq c_2
\end{align*}
with $c_2$ depending on
$\vre$ but not on $\mathbf{m}$.
On interchanging the
orders of integration and summation this implies that for almost every
$({\bf v}_1,\ldots , {\bf v}_d)\in R_1\times\cdots\times R_d$,
\begin{align}
S({\bf v}_1,\ldots , {\bf v}_d) \df \sum_{m_1=1}^\infty\cdots\sum_{m_d=1}^\infty
\prod_{i=1}^d\frac{1}{|{\bf m}\cdot {\bf v}_i|(\log
  m_i)^{2+\vre}(-\log \min(1/2,|{\bf m}\cdot {\bf
    v}_i|))^{1+\vre}}
\label{sumbnd1}
\end{align}
is finite and independent of $m_{d+1},m_{d+2},\ldots , m_k\in\Z$. Since the location of
the cubes $R_1,\ldots , R_d$ was arbitrary, $S({\bf v}_1,\ldots , {\bf
  v}_d) < \infty$
for almost every $({\bf v}_1,\ldots , {\bf
  v}_d)\in (\R^k)^d.$
By grouping together the choices for $m_{d+1}, \ldots, m_k$, we obtain
\begin{align}
&\sum_{\substack{{\bf m}\in\Z^k\\0<m_1,\ldots , m_k \le M}}
r_{\mathcal{T}} ({\bf m})\prod_{i=1}^d\frac{1}{|{\bf m}\cdot {\bf
    v}_i|}\nonumber\\
    &\nonumber\\
&\qquad\qquad\leq C(\log M)^{k-d}(\log M)^{d(2+\vre)}P(M)S({\bf
  v}_1,\ldots , {\bf v}_d),\label{sumbnd2}
\end{align}
where
\[P(M)=\prod_{i=1}^d\max_{1\le m_1,\ldots ,m_k\le
  M}(-\log \min(1/2,|{\bf m}\cdot {\bf v}_i|))^{1+\vre}.\]
In the inequality in (\ref{sumbnd2}) we are using the fact that for each $i\in\{d+1,\ldots , k\}$, the quantity $\mathbf{t}_i\cdot {\bf m}$ is always a fixed multiple of $m_j$ for some $j$.

By a standard application of the Borel-Cantelli Lemma,
for almost every ${\bf v}\in\R^k$ there is a constant $c=c({\bf
  v})>0$ such that
\begin{equation*}
|{\bf m}\cdot {\bf v}|\ge\frac{c}{M^{2k}}~\text{ for all }~{\bf
  m}\in\Z^k~\text{ with }~0<\|{\bf m}\| \le M.
\end{equation*}
Thus for almost every ${\bf v}_1,\ldots , {\bf v}_d$ and for any
$\delta>0$ we have that (\ref{sumbnd2}) is bounded above by a constant
times $(\log M)^{k+2d+\delta}$.

Finally we can estimate
\[\sum_{0<\|{\bf m}\| \le M}r_{\mathcal{T}} ({\bf m})\prod_{i=1}^d\frac{1}{|{\bf m}\cdot {\bf v}_i|}\]
by partitioning the sum into $2^k$ subsets of points ${\bf m}$,
according to which components of ${\bf m}$ are $0$. To each one of
these subsets we may then apply the above arguments to obtain the
required bound.
\end{proof}

As a corollary of our proof, the conclusions of Theorems \ref{thm:
  equidistribution for BL} and \ref{thm: equidistribution for BDD} can
be considerably strengthened, as follows. 

\begin{prop} \name{prop: strengthening}
For almost every $\mathbf{v}_1, \ldots, \mathbf{v}_d$,
and any basis $\mathcal{T}$ as in Proposition \ref{prop: required bound}, for any $\delta >0$ there is $c>0$ so that
\eq{hsumeqn1}{
\sum_{0<\|{\bf m}\| \le M}r_{\mathcal{T}}({\bf m})\prod_{i=1}^d
\frac{1}{|{\bf m}\cdot {\bf v}_i|}\leq c\, (\log M)^{k+2d+\delta}.
}
Under this condition, the error terms on the right hand sides of
\equ{eq: equidistribution} and \equ{eq: equidistribution2} can be
replaced by $C(\log T)^{k+2d+\delta}$. 
\end{prop}

\begin{proof}
The bound \equ{hsumeqn1} was already proved above. For the rest of the claim, take $M=T^d$ and use \equ{hsumeqn1} and
Proposition \ref{prop: another bound}  in Theorem \ref{thm:NTestimate}.
\end{proof}

To conclude this section we mention the following easy fact:
\begin{prop}\name{prop: Diophantine and strongly}
If $\mathbf{v}_1, \ldots, \mathbf{v}_d$ are strongly Diophantine then
each $\mathbf{v}_i$ is Diophantine.
\end{prop}

\begin{proof}
Suppose that $\mathbf{v}_1, \ldots, \mathbf{v}_d$ are strongly Diophantine with respect to $\mathcal{T} = (\mathbf{t}_1, \ldots,
\mathbf{t}_k)$, let $s>k+d-1$, and let $i \in \{1, \ldots, d\}$. Suppose by contradiction that there are infinitely
many vectors $\mathbf{m} \in \Z^k$ so that $|\mathbf{m} \cdot
\mathbf{v}_i| < \frac{1}{\|\mathbf{m}\|^s}.$ If $\mathbf{m}$ is one such vector then setting $M =
\|\mathbf{m}\| $ and using Cauchy-Schwarz we find, for each $j \neq i$,
$$
|\mathbf{m} \cdot \mathbf{v}_j | \leq M \|\mathbf{v}_j\|.
$$
Noting that
$
r_{\mathcal{T}}(\mathbf{m}) \geq \prod_{i=1}^k \frac{1}{\|\mathbf{t}_i\| \cdot \|\mathbf{m}\|}
$
gives
$$
r_{\mathcal{T}}(\mathbf{m}) \prod_{i=1}^d \frac{1}{|\mathbf{m} \cdot \mathbf{v}_i|}
\geq
M^{-k} \, \left(\prod_{i=1}^k \frac{1}{\|\mathbf{t}_i\|}\right)
\left(\prod_{j \neq i}\frac{1}{M\|
    \mathbf{v}_j\|}\right) \|\mathbf{m}\|^s\ge CM^{s-k-d+1}.
$$
This holds along a sequence of $M\rar\infty$. However for some $\vre
>0$ this contradicts \equ{eq: for strongly Diophantine}.
\end{proof}

\section{Proofs of Theorem \ref{thm: main, BL} and
  \ref{thm: main, BDD}(1),(2) }
\name{section: proof BL}

\begin{proof}[Proof of Theorem \ref{thm: main, BL}]
Let $V$ be a Diophantine
subspace, and let
$\mathbf{v}_1, \ldots, \mathbf{v}_d$ be a basis for $V$. Let
$\mathcal{S}$ be a section which is $(k-d)$-dimensionally open and
bounded, with $\dim_M \partial \mathcal{S} < k-d$. In light of
Corollary \ref{cor: wnlg}, and since $\dim_M$ does not change upon replacing
$\mathcal{S}$ by its image under a bilipschitz map, there is no loss
of generality in assuming that $\mathcal{S}$ is linear. Let
$\mathcal{U}_1$ be a closed ball around $0$ in $V$,
satisfying (i) of \S \ref{subsection: sections}, and define $W$ via
\equ{eq: defn V}. Then $W$ is bi-Lipschitz equivalent to $\mathcal{U}
\times \mathcal{S}$ and hence, by \cite[Formulae 7.2 and 7.3]{Falconer}, $\dim_M \partial W <
k$. Thus the Theorem follows from Corollaries \ref{cor: dynamical for BL} and \ref{cor: for
main, BL}.
\end{proof}

\begin{proof}[Proof of Theorem \ref{thm: main, BDD}(1)]
Let $\mathbf{v}_1, \ldots, \mathbf{v}_d$ satisfy the conclusion of
Proposition \ref{prop: required bound}, and for $i=d+1, \ldots, k$, let
$\mathbf{v}_i \in
\{\mathbf{e}_1, \ldots, \mathbf{e}_k\} $ such that $\mathcal{T}=(\mathbf{v}_1,
\ldots, \mathbf{v}_k)$ is a basis of $\R^k$.  Also 
let $V = \spa (\mathbf{v}_1, \ldots, \mathbf{v}_d)$.  We need to show that for any
linear section $\mathcal{S}$ in a space $L$ transverse to $V$, such that
$\dim \partial \, S = k-d-1$, and any
$\bx \in \T^k$, the corresponding net is BD to a lattice. To this end we
will apply Corollaries \ref{cor: dynamical for BDD} and \ref{cor: for
  main, BDD} . 
Let $B$ be a ball in $L$ such that $\pi$ is injective on $B$, and sets
$\mathcal{U}_1$ and $\mathcal{U}_2$ satisfying conditions (i) and
(ii) of \S\ref{subsection:
  sections} for $B'$. Also let $L' \df \spa (\mathbf{v}_{d+1},
\ldots, \mathbf{v}_k)$, and let $B'$ be a ball in $L'$ such that
$\pi$ is injective on $B'$. Then $B'$ is a good section, let $\mathcal{U}'_1, \mathcal{U}'_2$ be
the corresponding sets as in \S \ref{subsection: sections}.

Suppose first that  $B$ is
small enough so that \equ{eq: hypothesis 1} holds. Then we can assume
with no loss of generality that $\mathcal{S}$ is contained in
$B'$. This in turn shows that the hypotheses of Corollaries \ref{cor:
  for main, BDD} and \ref{cor: dynamical
  for BDD} are satisfied, and $Y$ is BD to a lattice. 

Now suppose \equ{eq: hypothesis 1} does not hold. Then we can
partition $\mathcal{S}$ into smaller 
sets $\mathcal{S}^{(1)},\ldots , \mathcal{S}^{(r)}$ with equal
volume and $\dim_M \partial \mathcal{S}^{(i)}=k-d-1$, such that the
corresponding sets $\mathcal{U}_1^{(i)}$ satisfy 
\equ{eq: hypothesis 1}. 
Now repeating the previous argument separately to each
$\mathcal{S}^{(i)}$, we see that the corresponding net is BD to a
fixed lattice $L$. Note that the lattice is the same because each
$\mathcal{S}_i$ has the same volume. Now the result follows via 
Proposition \ref{prop: replace with rational subspace}.
\end{proof}

\begin{proof}[Proof of Theorem \ref{thm: main, BDD}(2)]
Suppose
$\mathcal{S}$ is a box with sides parallel to the
coordinate axes; that is, there is $J \subset \{1, \ldots, k\}, \, |J|
= k-d$, such that $\mathcal{S}$ is the projection under $\pi$ of an
aligned box in the space $V_J \df \spa (\E_j: j \in J)$.
As above, we can use Proposition \ref{prop: replace with rational
  subspace} to assume that $\pi$ is injective on a subset of $V_J$
covering $\mathcal{S}$. 
According to Proposition \ref{prop: required bound}, for almost every
choice of $\mathbf{v}_1, \ldots, \mathbf{v}_d$, the space $V = \spa
(\mathbf{v}_i)$ is strongly Diophantine with respect to the basis
$$\mathcal{T} \df \{\mathbf{v}_i: i=1, \ldots, d\} \cup
\{\E_j: j \in J\}.$$
As in the
preceding proof, choose a neighborhood $\mathcal{U}_1$
of $0$ in $V$ satisfying property (i) of \S\ref{subsection:
  sections} which is a box. Then  the set $W$ defined by
\equ{eq: defn V} is a parallelotope aligned with $\mathcal{T}$. 
According to Theorem \ref{thm: equidistribution for BDD}, \equ{eq: need to
  check, BDD} holds, and we can apply Corollary \ref{cor: dynamical for BDD}.
\end{proof}

\section{Irregularities of distribution}\name{section: irregularities}
In this section we will fix $1< d <k$ and let $\mathcal{G}$ denote the Grassmannian
variety of $d$-dimensional subspaces of $\R^k$. We denote by
$\mc{G}(\Q)$ the subset of rational subspaces.
We will fix a totally irrational $k-d$ dimensional subspace $W \subset
\R^k$, and let $\mathcal{S}$ be the image under $\pi$ of a subset of
$W$ which is open and bounded. There is a dense $G_\delta$ subset of
$V \in \mathcal{G}$ for which $\mc{S}$ is a good section for the
action of $V$ on $\T^k$; indeed, by the discussion of
\S\ref{subsection: sections},  this holds whenever $V$ and $W$ are
transverse to each other and $V$ is totally irrational.

If $Q \in \mc{G}(\Q)$ then any orbit  $Q.\bx$ is compact; further if $Q$ is
transverse to $W$ then $Q.\bx \cap \mathcal{S}$ is a finite set
for every $\bx \in \T^k$.
We say that $\mathcal{S}$ and $Q$ are  {\em not correlated} if there
are $\bx_1, \bx_2 \in \T^k$ such that
\eq{eq: defn not correlated}{
\# \, \left(Q.\bx_1 \cap \mc{S} \right)  = \# \, \left(Q.\bx_1 \cap
  \overline{\mc{S}} \right) \neq \# \, \left( Q.\bx_2 \cap
\mc{S}\right) = \# \, \left(Q.\bx_2 \cap \overline{\mc{S}} \right)
}
(here $\overline{\mc{S}}$ denotes the closure of $\mc{S}$).
We say that $\mathcal{S}$
is {\em typical} if there is a dense set of $Q \in \mc{G}$ for which
$\mc{S}$ and $Q$ are not correlated.

It is not hard to find typical $\mathcal{S}$:
\begin{prop}\name{prop: bad rational dense}
Let $r=k-d$ and let $W$ be a totally irrational $r$-dimensional
subspace of $\R^k$. Let
$\mathbf{w}_1, \ldots, \mathbf{w}_r$ be a basis for $W$ and for $\A = (a_1, \ldots,
a_r) \in (0,1)^r, \B=(b_1, \ldots, b_r) \in (0,1)^r$ let
$$P(\A, \B) \df
\pi \left(\left\{\sum_1^r  t_i \mathbf{w}_i : t_i \in (a_i, a_i+b_i)
\right\}\right).
$$
Then the set of $(\A, \B)$ for which $P(\A, \B)$ is not correlated
with any rational subspace, and hence typical, is of full measure and residual in
$(0,1)^{2r}.$

\end{prop}
\begin{proof}
It is enough to show that for a fixed $Q$, the set of $\A, \B$ for
which $P(\A, \B)$ is correlated with $Q$ has zero measure and is a
submanifold of dimension less than $2r$ in $[0,1]^{2r}$. To see this,
define two functions $F, \overline{F}$ on $\T^k$, by
$$
F(\bx) \df  \# \, \left( Q.\bx \cap \mc{S}\right), \ \
\overline{F}(\bx) \df \#  \, \left( Q.\bx \cap \overline{\mc{S}} \right).
$$
We always have $F(\bx) \leq \overline{F}(\bx)$, and $F(\bx) =
\overline{F}(\bx)$ unless $Q.\bx$ intersects the boundary of
$\mc{S}$. So if \equ{eq: defn not correlated} fails then $F(\bx)$
always has the same value, for the values of $\bx$ for which $Q.\bx
\cap \partial \, \mc{S} = \varnothing.$

Note that the values of $F, \overline{F}$ are constant along orbits of
$Q$. The space of orbits for the $Q$-action is itself a compact torus
$Q'$ of dimension $r
$. Let $\pi' : \T^k \to Q'$ be the projection
mapping a point to its orbit.
The discussion in the previous paragraph
shows that the requirement that $\mathcal{S}$ and $Q$ are correlated
is equivalent to the requirement that the interior of $\mc{S}$
projects onto a dense open subset of $Q'$ with fibers of constant cardinality. Clearly this
property is destroyed if we vary $\mc{S}$ slightly in the direction
orthogonal to $Q$. More precisely, for any $\A$ and $\B$, there
is a small neighborhood $\mathcal{U}$ such that which the set of $\A',
\B'$ in $\mathcal{U}$ for which \equ{eq: defn not correlated} fails is
a proper submanifold of zero measure. This proves the claim.


\combarak{I suspect that all
  boxes and balls are typical.
By the preceding discussion this is the same as saying that for fixed
$\mc{S}$, for a dense set of $Q$, the projection of $\mc{S}$ onto the
torus $Q'$ of $Q$-orbits, is essentially non-constant. I am sure this
is true but can you prove it?}
\end{proof}

By similar arguments one can show that almost every ball, ellipsoid,
etc., is typical.
\begin{prop}\name{prop: Baire}
If $\mc{S}$ is a bounded open set whose boundary is of zero measure
(w.r.t. the Lebesgue measure on the subspace $W$), and $\mc{S}$ is typical, then there is a dense $G_\delta$ subset of $V$ for which,
for every $\bx \in \T^k$,
the separated net $Y_{\mc{S},\bx}$ is not BD to a lattice.
\end{prop}

\begin{proof}
Let $Q_1, Q_2, \ldots$ be a list of rational subspaces in
$\mc{G}(\Q)$ such that
$\mc{S}$ and $Q_i$ are not correlated for each $i$, and $\{Q_i\}$ is a
dense subset of $\mc{G}$. For each $i$ let $\bx^{(i)}_1, \bx^{(i)}_2$
be two points in $\T^k$ for which \equ{eq: defn not correlated}
holds.
Since the linear action of subspaces on $\T^k$ is the
restriction of the continuous natural $\R^k$-action, for any
$\vre>0$ and any $T>0$ we can find a neighborhood of $Q_i$ in $\mc{G}$
consisting of subspaces $V$ such that for any ${\bf v} \in V$ with
$\|{\bf v}\|<T$, and any $\bx \in \T^k$, the distance in $\T^k$ between
${\bf v}.\bx$ and ${\bf v}'.\bx$ is less than $\vre$, where ${\bf v}'$ is the orthogonal
projection of ${\bf v}$ onto $Q_i$. We will fix below a sequence of bounded
sets $M_i \subset Q_i$ and denote by $M_i^{(V)}$ the preimage, under
orthogonal projection $V \to Q_i$, of
$M_i$. Using our assumption on $\mc{S}$, by perturbing $\bx^{(i)}_1, \bx^{(i)}_2$ slightly we can assume
that ${\bf q}.\bx_1^{(i)}$ and ${\bf q}.\bx^{(i)}_2$ are not in $\partial \mc{S}$
when ${\bf q} \in M_i$.
Since $\mc{S}$ is relatively open in
$W$, this implies that there is an open
subset $\mathcal{V}_i$ of $\mc{G}$
containing $Q_i$, such that for every $V \in \mathcal{V}_i$ and for
$\ell=1,2$,
\eq{eq: make sure}{
\#\left\{
{\bf q} \in M_i:  {\bf q}.\bx_{\ell}^{(i)} \in \mathcal{S}
\right\}
=\# \left\{
{\bf v} \in M_i^{(V)}:  {\bf v}.\bx_{\ell}^{(i)} \in \mathcal{S}
\right\}.
}
Then
$$\mc{G}_{\infty} \df \bigcap_{i_0} \bigcup_{i \geq i_0}
\mathcal{V}_i$$
is clearly
a dense $G_\delta$ subset of $\mc{G}$, and it remains to show that by
a judicious choice of the sequence $M_i$, we can ensure that for any
totally irrational $V
\in \mc{G}_{\infty}$, for any $\bx$, and any positive $\lambda, c$,
the separated net $Y_{\mc{S},\bx}$ does not
satisfy condition (3) of Theorem \ref{thm: Laczkovich2}.

For any $i$ let $C_i$ be a parallelotope which is a fundamental domain for the action of the lattice $Q_i
\cap \Z^k $ on
$Q_i$. Specifically we let
$$C_i \df \left\{\sum_{j=1}^d a_j \mathbf{q}_j:
\forall j, \, 0 \leq a_j <  \|\mathbf{q}_j\|  \right \},
$$ where
$\mathbf{q}_1, \ldots, \mathbf{q}_d$ are
a basis of $Q_i
\cap \Z^k$.
 We claim that there are  positive
constants $c_1, c_2, C$ (depending on $i$) and sets $M_i$
which are finite unions of translates of $C_i$, of arbitrarily large
diameter, such that:
\eq{eq: first inequality}{
|M_i| \geq c_1 \diam(M_i)^d;
}
\eq{eq: second inequality}{
\left| (\partial M_i)^{(1)} \right| \leq C\, \diam(M_i)^{d-1}
}
 (where, as before,  $(\partial M_i)^{(1)}$
is the set of points at distance 1 from $\partial M_i$).
Indeed, we simply take $M_i$ to be dilations by an integer factor, of
$C_i$ around its center. Then each $M_i$ is homothetic to $C_i$ and
\equ{eq: first inequality} and \equ{eq: second inequality} follow.
Now let
$N_i$ be the number of copies of $C_i$ in $M_i$. Then
$$\# \{q
\in M_i: q.\bx^{(i)}_\ell \in \mc{S} \}  = N_i \cdot \#
\{
q \in C_i: q.\bx^{(i)}_\ell \in \mc{S} \} = N_i \cdot \# \, \left(
  Q. \bx^{(i)}_{\ell} \cap \mc{S}\right)
$$
and
$$
|M_i| = N_i \cdot |C_i|,
$$
which implies via \equ{eq: first inequality} and \equ{eq: second inequality} that for some constant $c_2$,
$$
|(\partial M_i)^{(1)}| \leq c_2 N_i^{1-1/d}.
$$
If we set
$$c_3 \df \frac{
\left | \# \left(
 Q. \bx^{(i)}_2\cap \mc{S}\right) -  \# \left( Q. \bx^{(i)}_1 \cap \mc{S}\right)
 \right|
}{2}
,$$
then for any $\lambda$, there is $\ell \in \{1,2\}$ such that for $\bx' =
\bx_{\ell}^{(i)}$ we have
$$
\left| \#(Q. \bx' \cap \mc{S}) - \lambda |C_i| \right| \geq c_3,
$$
and hence
$$
\frac{\left |\# (M_i.\bx' \cap \mc{S}) -\lambda |M_i|\right |}{|(\partial
  M_i)^{(1)}|} \geq \frac{N_i\left|\#(Q \bx'\cap \mc{S}) -
    \lambda|C_i|\right|}{c_2N_i^{1-1/d}} \geq \frac{c_3}{c_2} N_i^{1/d}.
$$
So by choosing $N_i$ large enough we can ensure that for any
$\lambda$, and $\bx'$ one of the $\bx^{(i)}_{\ell}$, we have
\eq{eq: even surer}{
\left |\# (M_i.\bx' \cap \mc{S}) -\lambda |M_i|\right |\geq i \, |(\partial
  M_i)^{(1)}|.
}
Now fixing $\lambda$ and $c$ we choose $i > c$ and choose $\bx'$ as
above depending on $\lambda$.
If $V \in \mc{V}_i$ is totally irrational then for any $\bx \in \T^k$ there is a sequence
$v_n \in V$ such that $v_n.\bx\to \bx'$. So we may replace $\bx'$ with
$\bx$ and $M_i$ with $v_n+M_i$ for sufficiently large $n$, and
\equ{eq: even surer} will continue to hold.
In light of \equ{eq: make sure}, if $Y$ is the net corresponding to
$V, \, \mc{S}$ and $\bx$, and $E \df v_n+M_i$, then we have shown
$\Disc_Y(E, \lambda) > c |(\partial E)^{(1)}|$,
and we have a contradiction to
condition (3) of Theorem \ref{thm: Laczkovich2}.
\end{proof}

\begin{proof}[Proof of Theorem \ref{thm:  main, BDD}(3)]
Immediate from Propositions \ref{prop: bad rational dense} and \ref{prop: Baire}.
\end{proof}
\ignore{
\section{The Penrose net revisited}\name{section: penrose}

{\comment Hopefully reprove Ya'ar's result that the Penrose tiling is
  BDD to a lattice, using Theorem \ref{thm: main, BDD}(2) and results
  of \S \ref{subsection: tilings}. }
}

\bibliographystyle{abbrv}
\bibliography{SeparatedNets_bibliography}

\end{document}